\documentclass[11pt]{amsart}

\usepackage{amsmath, amsthm, amssymb}
\usepackage{amsfonts}
\usepackage{esint}
\usepackage[ansinew]{inputenc}
\usepackage[dvips]{epsfig}
\usepackage{graphicx}
\usepackage[english]{babel}
\usepackage{pdfsync}
\usepackage{hyperref}
\usepackage{color}
\usepackage{array}
\usepackage{appendix}

\usepackage[
  hmarginratio={1:1},     % equal left and right margins
  vmarginratio={1:1},     % equal top and bottom margins
  textwidth=17cm,        % new text width
  textheight=21cm,
  heightrounded,          % always useful
]{geometry}

\newtheorem{thm}{Theorem}[section]
\newtheorem{cor}[thm]{Corollary}
\newtheorem{lem}[thm]{Lemma}
\newtheorem{prop}[thm]{Proposition}

\newtheorem{defn}[thm]{Definition}
\newtheorem{rem}[thm]{Remark}

\newtheorem*{thm*}{Theorem}
\newtheorem*{defn*}{Definition}
\numberwithin{equation}{section}

\newcommand{\dx}{\,{\rm d}x}
\newcommand{\dX}{\,{\rm d}X}

\newcommand{\dz}{\,{\rm d}z}
\newcommand{\dZ}{\,{\rm d}Z}

\newcommand{\dt}{\,{\rm d}t}
\newcommand\dtau{{\,{\rm d}\tau}}
\newcommand{\rd}{{\rm d}}

\newcommand{\BB}{\mathbb{B}}

\newcommand{\NN}{\mathbb{N}}
\newcommand{\QQ}{\mathbb{Q}}

\newcommand{\RR}{\mathbb{R}}
\newcommand\EE{{\mathcal{E}}}
\newcommand\FF{{\mathcal{F}}}
\newcommand\JJ{{\mathcal{J}}}
\newcommand\UU{{\mathcal{U}}}

\newcommand{\ep}{\epsilon}
\newcommand{\vep}{\varepsilon}
\newcommand{\dd}{\delta}

\def\I{\mathrm{I}}

\def\R{\mathrm{R}}
\def\E{\mathrm{E}}

\def\supp{\mathrm{supp}} %per il supporto
\DeclareMathOperator*{\osc}{osc}
\DeclareMathOperator*{\esssup}{ess\,sup}
\DeclareMathOperator*{\essinf}{ess\,inf}

\begin{document}

\title[]{On the existence and H\"older regularity of solutions to some nonlinear Cauchy-Neumann problems}

\author[A. Audrito]{Alessandro Audrito}%\thanks{}
\address{Alessandro Audrito \newline \indent
ETH Z\"urich, D-Math, R\"amistrasse 101, 8092 Zurich, Switzerland.
 }
\email{alessandro.audrito@math.ethz.ch, alessandro.audrito@gmail.com}

\date{\today} %%  this cancels date in article format


%
%
%
%
%
%
%%%%%%%%%%%%%%%%%%%%%%%%%%%%%%%%%%%%%%%%%%%%%%%%%%%%%%%%%%%%%%%%%%%%%%%%%%%%%%%%%%%%%%%%%%%%%%%%%%%%%%%%
%%%%%%%%%%%%%%%%%%%%%%%%%%%%%%%%%%%%%%%%%%%%%%%%%%%%%%%%%%%%%%%%%%%%%%%%%%%%%%%%%%%%%%%%%%%%%%%%%%%%%%%%
%
%
%
%
%
%%%%%%%%%%%%%%%%%%%%%%%%%%%%%%%%%%%%%%%%%%%%%%%%%%%%%%%%%%%%%%%%%%%%%%%%%%%%%%%%%%%%%%%%%%%%%%%%%%%%%%%%

\thanks{This project has received funding from the European Union's Horizon 2020 research and innovation programme under the Marie Sk{\l}odowska--Curie grant agreement 892017 (LNLFB-Problems). The author wishes to thank Prof. Matteo Bonforte and Prof. Paolo Tilli for inspiring discussions concerning the results presented in this paper.
}

\begin{abstract}
We prove \emph{uniform} parabolic H\"older estimates of De Giorgi-Nash-Moser type for sequences of minimizers of the functionals
\[
\EE_\vep(W) = \int_0^\infty \frac{e^{- t/\varepsilon}}{\varepsilon} \bigg\{ \int_{\mathbb{R}_+^{N+1}} y^a \left( \vep|\partial_t W|^2 + |\nabla W|^2 \right)\rd X + \int_{\mathbb{R}^N \times\{0\}} \Phi(w) \dx \bigg\}\dt, \qquad \vep \in (0,1) 
\]
where $a \in (-1,1)$ is a fixed parameter, $\mathbb{R}_+^{N+1}$ is the upper half-space and $\rd X = \rd x \rd y$. As a consequence, we deduce the existence and H\"older regularity of weak solutions to a class of weighted nonlinear Cauchy-Neumann problems arising in combustion theory and fractional diffusion. 
\end{abstract}

\maketitle

%\tableofcontents

%%%%%%%%%%%%%%%%%%%%%%%%%%%%%%%%%%%%%%%%%%%%%%%%%%%%%%%%%%%%%%%%%%%%%%%%%%%%%%%%%%%%%%%%%%%%%%%%%%%%%%%%%%%%%

%%%%%%%%%%%%%%%%%%%%%%%%%%%%%%%%%%%%%%%%%%%%%%%%%%%%%%%%%%%%%%%%%%%%%%%%%%%%%%%%%%%%%%%%%%%%%%%%%%%%%%%%%%%%%
%
%
%
%
\section{Introduction}
In this paper we construct \emph{H\"older continuous} weak solutions to the weighted nonlinear Cauchy-Neumann problem
\begin{equation}\label{eq:ParReacDiff}
\begin{cases}
y^a\partial_t U - \nabla\cdot (y^a \nabla U) = 0 \quad &\text{ in }  \mathbb{R}^{N+1}_+ \times (0,\infty) \\
-\partial_y^a U = - \beta(u)                                \quad &\text{ in } \mathbb{R}^N \times \{0\} \times (0,\infty) \\
U|_{t=0} = U_0     \quad &\text{ in } \mathbb{R}^{N+1}_+,
\end{cases}
\end{equation}
where $N \geq 1$, $a \in (-1,1)$, $\RR_+^{N+1} := \{X=(x,y): x\in \RR^N, y > 0\}$, $\nabla$ and $\nabla\cdot$ stand for the gradient and the divergence operators w.r.t. $X$, respectively, and
\[
u := U|_{y=0} \qquad \text{ and } \qquad  \partial_y^a U := \lim_{y\to0^+}y^a\partial_y U.
\]
The weight $y^a$ belongs to the Muckenhoupt $A_2$-class (cf. \cite{ChiarenzaSerapioni1985:art,FabesKS1982:art}), the function $U_0$ is a given initial data  and $\beta \in C(\RR;\RR)$ is of combustion type, satisfying
\begin{equation}\label{eq:Reaction}
\beta \geq 0, \qquad \supp \beta = [0,1], \qquad \int_0^1\beta(v)\,\rd v = \tfrac{1}{2}.
\end{equation}
Problem \eqref{eq:ParReacDiff} is related to the \emph{localized/extended} version of the reaction-diffusion equation 
\begin{equation}\label{eq:NonlocalEqBeta}
(\partial_t - \Delta)^su = - \beta(u),
\end{equation}
where $s := \tfrac{1-a}{2} \in (0,1)$ (see \cite{NystromSande2016:art,StingaTorrea2015:art} and \cite[Section 2]{BanGarofalo2017:art}), and the diffusion process is governed by the fractional power of the heat operator, which is nonlocal both in space and time:
\[
(\partial_t - \Delta)^su(x,t) = \frac{1}{|\Gamma(-s)|} \int_{-\infty}^t \int_{\RR^N} \left[u(x,t) - u(z,\tau)\right] \frac{G_N(x-z,t-\tau)}{(t-\tau)^{1+s}} \dz \rd \tau,
\]
where $G_N$ is the fundamental solution to the heat equation and $\Gamma$ is the gamma function (cf. \cite[Section 28]{SamkoKilbasMarichev:1993}). For smooth functions $u$ depending only on the space variables $x\in\RR^N$, it reduces to the fracional laplacian $(-\Delta)^s$, while if $u = u(t)$, it is the Marchaud derivative $(\partial_t)^s$ (cf. \cite{Marchaud1927:art,StingaTorrea2015:art}).

Such operator appears in a wide range of applications such as biology, physics and finance (see e.g. \cite{AthCaffaMilakis2016:art,DainelliGarofaloPetTo2017:book} and the monograph \cite{Hilfer2000:book}) and has notable interpretations in Continuous Time Random Walks theory (see \cite{MetzlerKlafter2000:art} and the references therein). In recent years this class of equations has been the subject of intensive research: we quote \cite{CabreEtAl2015:art,CafMelletSire2012:art} for traveling wave analysis, \cite{BanGarofalo2017:art} and \cite{AthCaffaMilakis2016:art,BanGarDanPetr2020:art,BanGarDanPetr2021:art,DainelliGarofaloPetTo2017:book} for unique continuation and obstacle problems, and \cite{AudritoTerracini2020:art} for nodal set analysis. In our context, solutions to \eqref{eq:NonlocalEqBeta} may be employed to approximate some \emph{free boundary} problems arising in combustion theory and flame propagation, in the singular limit $\beta \to \tfrac{1}{2}\dd_0$ (cf. \cite{CafRoqueSire2010:art,PetShiSire2016:art} for the nonlocal elliptic framework and \cite{CafVaz1995:art} for the local parabolic one).

In this work, weak solutions to \eqref{eq:ParReacDiff} will be obtained through a \emph{variational} approximation procedure known in the literature as the \emph{Weighted Inertia-Energy-Dissipation} method, introduced in the works of Lions \cite{Lions1965:art} and Oleinik \cite{Oleinik1964:art} (see also the paper of De Giorgi \cite{DeGiorgi1996:art} in the context of nonlinear wave equations). Later, it has been investigated by many authors: we quote the works of Akagi and Stefanelli \cite{AkagiStefanelli2016:art}, Mielke and Stefanelli \cite{MielkeStefanelli2011:art}, B\"ogelein et al. \cite{BogeleinEtAl2014:art,BogeleinEtAl2017:art} and the references therein. However, the variational techniques we use are inspired by the methods developed by Serra and Tilli in \cite{SerraTilli2012:art,SerraTilli2016:art} (see also the more recent \cite{AudritoSerraTilli2021:art}). We introduce below the main ideas in a rather informal way, and postpone the formal definitions and statements in subsequent sections.

We set 
\[
\Phi(u) := 2 \int_0^u\beta(v)\, \rd v,
\]
and, for every fixed $\vep \in (0,1)$, we introduce the functional
\begin{equation}\label{eq:Functional0}
\EE_\vep(W) = \int_0^\infty \frac{e^{- t/\varepsilon}}{\varepsilon} \bigg\{ \int_{\mathbb{R}_+^{N+1}} y^a \left( \vep|\partial_t W|^2 + |\nabla W|^2 \right)\rd X + \int_{\mathbb{R}^N \times\{0\}} \Phi(w) \dx \bigg\}\dt.
\end{equation}
If $\EE_\vep$ has a minimizer $U_\vep$ (or a minimizer pair $(U_\vep,u_\vep:=U_\vep|_{y=0})$) in some suitable space $\UU_0$ (see \eqref{eq:Ass2Data}) with $U_\vep|_{t=0} = U_0$, then $U_\vep$ satisfies
\begin{equation}\label{eq:PerProbEps}
\begin{cases}
-\varepsilon y^a\partial_{tt} U_\vep + y^a\partial_t U_\vep - \nabla\cdot (y^a \nabla U_\vep) = 0 \quad &\text{ in }  \mathbb{R}^{N+1}_+ \times (0,\infty) \\
-\partial_y^a U_\vep = - \beta(u_\vep)                               \quad &\text{ in } \mathbb{R}^N \times \{0\} \times (0,\infty) \\
U_\vep|_{t=0} = U_0     \quad &\text{ in } \mathbb{R}^{N+1}_+,
\end{cases}
\end{equation}
in the weak sense (see Lemma \ref{Lem:EulerEqMinFF}). The above problem is nothing more than \eqref{eq:ParReacDiff} with the extra term $-\varepsilon y^a\partial_{tt} U_\vep$: it is thus reasonable to conjecture that under suitable boundedness and compactness properties, one may pass to the limit as $\vep \to 0$ along some subsequence and obtain a limit weak solution $U$ to \eqref{eq:ParReacDiff}.

We stress that for each $\vep \in (0,1)$ the approximating problem \eqref{eq:PerProbEps} is \emph{elliptic} in space-time and the drift $y^a\partial_t U_\vep$ is a lower order term, while, as $\vep \to 0$, it degenerates along the time direction and, in the limit, the problem completely changes nature, becoming \emph{parabolic}.

As already mentioned, our main goal is to establish uniform estimates for families of minimizers of the functional $\EE_\vep$ and then pass to the limit by compactness. We will deal with two types of uniform bounds: global energy estimates and H\"older estimates. The former are obtained adapting the techniques of \cite{SerraTilli2012:art,SerraTilli2016:art}, while the latter will follow from a De Giorgi-Nash-Moser type result (\cite{DeGiorgi1957:art,Nash1958:art,Moser1964:art}) for weak solutions to 
\begin{equation}\label{eq:PerProbEpsFf}
\begin{cases}
-\varepsilon y^a\partial_{tt} W_\vep + y^a\partial_t W_\vep - \nabla\cdot (y^a \nabla W_\vep) = F_\vep \quad &\text{ in }  \mathbb{R}^{N+1}_+ \times \RR \\
-\partial_y^a W_\vep = f_\vep,                               \quad &\text{ in } \mathbb{R}^N \times \{0\} \times \RR, 
\end{cases}
\end{equation}
where $F_\vep$ and $f_\vep$ belong to suitable classes of spaces (see Appendix \ref{App:WeightedSpaces} and \eqref{eq:AssumptionsExponents}).

This is our main contribution: we prove \emph{parabolic} H\"older estimates ``up to $\{y=0\}$'' for weak solutions to problem \eqref{eq:PerProbEpsFf}, that we transfer to sequences of minimizers of \eqref{eq:Functional0} and, in turn, to the limit function $U$, as $\vep \to 0$. We anticipate that even though these H\"older estimates have a \emph{local} nature (we work directly with the local weak formulation of \eqref{eq:PerProbEpsFf}), we will need an extra compactness assumption guaranteed by the global energy estimates, and depending on the initial data $U_0$ (see Proposition \ref{prop:weaklimit}).

Since problem \eqref{eq:PerProbEpsFf} is elliptic for every $\vep > 0$ but becomes parabolic in the limit $\vep =0$, we cannot expect to prove uniform \emph{elliptic} H\"older estimates (i.e. elliptic in the $(X,t)$ variables), but ``only'' \emph{parabolic} ones. To do this, we will combine uniform local energy estimates, uniform local $L^2 \to L^\infty$ bounds and a uniform oscillation decay lemma. It is important to stress that the proofs of these intermediate steps are not just the mere adaptation of the parabolic theory (see for instance \cite{Moser1964:art,Vasseur2016:art}), but are tailored to the degeneracy of the problem along the time direction. We refer the reader to Subsection \ref{subsec:MainResults} for further discussions and connections with the existing literature.

Finally, it is important to mention that our strategy seems to be quite \emph{flexible} and different parabolic problems may be attacked with similar techniques. For instance, one may fix $s \in (0,1)$ and try to approximate weak solutions to 
\[
\partial_t u + (-\Delta)^s u = -\beta(u) 
\]
with a sequence of minimizers of 
\[
\tilde{\EE}_\vep(w) = \int_0^\infty \frac{e^{- t/\varepsilon}}{\varepsilon} \bigg\{ \vep \int_{\mathbb{R}^N}|\partial_t w|^2\dx  +  \int_{\mathbb{R}^N \times \mathbb{R}^N} \frac{[w(x,t) - w(z,t)]^2}{|x-z|^{N+2s}} \dx \rd z + \int_{\mathbb{R}^N} \Phi(w) \dx \bigg\}\dt,
\]
over a suitable functional space (notice that here we work in the purely nonlocal framework, without making use of the extension theory for the fractional laplacian). In view of \cite{AkagiStefanelli2016:art,SerraTilli2016:art}, the techniques we use to prove the energy estimates should easily be adapted to this setting too, whilst the general strategy we follow to obtain the H\"older bounds seems to be more difficult to repeat and must be adapted depending on the different nature of the problem.

%
%
%%%%%%%%%%%%%%%%%%%%%%%%%%%%%%%%%%%%%%%%%%%%%%%%%%%%%%%%%%%%%%%%%%%%%%%%%%%%%%%%%%%%%%%%%%%%%%%%%%%%%%%%%%%%%%%%%%%%%%%%%%%%%%%%%%%%%%%%%%%%%%%%%%%%%
%
%
\subsection{Functional Setting}
To simplify the notation, we work with the functional
\begin{equation}\label{eq:Functional}
\FF_\vep(U) = \int_0^\infty \frac{e^{- t/\varepsilon}}{\varepsilon} \bigg\{ \int_{\mathbb{R}^{N+1}} |y|^a \left( \vep|\partial_t U|^2 + |\nabla U|^2 \right)\rd X + \int_{\mathbb{R}^N \times\{0\}} \Phi(u) \dx \bigg\}\dt,
\end{equation}
and then we will transfer the information to the minimizers of \eqref{eq:Functional0}, using standard even reflections w.r.t. $y$ (see Lemma \ref{lem:BasicPropMin} and Remark \ref{rem:EvenPropMin}). Indeed, notice that $\FF_\vep$ is nothing more than $\EE_\vep$ but the integration is on the whole $\RR^{N+1}$ and, as always, $u := U|_{y=0}$.

We consider the space
\[
\UU := \bigcap_{R > 0} H^{1,a}(\mathbb{Q}_R^+), \qquad \mathbb{Q}_R^+ := \mathbb{B}_R \times (0,R^2),
\]
made of functions $U \in L^{2,a}(\mathbb{Q}_R^+)$ with weak derivatives $\partial_t U \in L^{2,a}(\mathbb{Q}_R^+)$, $\nabla U \in (L^{2,a}(\mathbb{Q}_R^+))^{N+1}$, for every $R > 0$ (the definitions of the $L^{p,a}$ spaces are given in Appendix \ref{App:WeightedSpaces}).
In particular, by \cite{Nekvinda1993:art}, each function $U \in \UU$ has a trace on the hyperplane $\{y=0\}$ and an ``initial'' trace we denote with
\[
u := U|_{y=0} \in H^{\frac{1-a}{2}}_{loc}(Q_\infty), \qquad U_0 := U|_{t=0} \in H^{\frac{1-a}{2}}_{loc}(\mathbb{R}^{N+1}),
\]
respectively ($Q_\infty := \RR^N\times\{0\}\times(0,\infty)$).

Since each term in \eqref{eq:Functional} is nonnegative, we will view $\FF_\vep$ as a functional defined on $\UU$ taking values in $[0,+\infty]$ and we will minimize it on the space $\UU$, subject to the initial condition
\[
U|_{t=0} = U_0, \qquad U_0 \in H^{1,a}(\mathbb{R}^{N+1}).
\]
The choice of $U_0 \in H^{1,a}(\mathbb{R}^{N+1})$ is quite important for our approach: we have $U_0 \in \UU$ (and thus $\UU$ is not empty) and, by \cite{Nekvinda1993:art} again, we also have that
\[
u_0 := U_0|_{y=0} \in H^{\frac{1-a}{2}}(\mathbb{R}^N)
\]
is well-defined. To prove our main estimates, it will be crucial to assume
\[
\mathcal{L}^N(\{ u_0 > 0 \}) < +\infty,
\]
where $\mathcal{L}^N$ denotes the $N$-dimensional Lebesgue measure. This is not a very restrictive assumption: in the majority of the applications the function $u_0$ is assumed to be smooth and compactly supported in $\mathbb{R}^N$. To simplify the notations, it is convenient to introduce the set of initial traces
\begin{equation}\label{eq:TraceSpace}
\mathcal{T}_0 := \big\{ U_0 \in H^{1,a}(\mathbb{R}^{N+1}): U_0 \not\equiv 0, \, \mathcal{L}^N(\{ u_0 > 0 \}) < +\infty  \big\},
\end{equation}
and the (non-empty) closed convex linear space
\begin{equation}\label{eq:Ass2Data}
\UU_0 := \{ U \in \UU: U|_{t=0} = U_0 \in \mathcal{T}_0 \}.
\end{equation}
Notice that the assumptions on the initial trace guarantee that any minimizer $U \in \UU_0$ is nontrivial.

%
%
%%%%%%%%%%%%%%%%%%%%%%%%%%%%%%%%%%%%%%%%%%%%%%%%%%%%%%%%%%%%%%%%%%%%%%%%%%%%%%%%%%%%%%%%%%%%%%%%%%%%%%%%%%%%%%%%%%%%%%%%%%%%%%%%%%%%%%%%%%%%%%%%%%%%%
%
%
\subsection{Main results}\label{subsec:MainResults} Our main result is the following theorem.
\begin{thm}\label{thm:MainIntro}
Let $N \geq 1$, $a \in (-1,1)$, $\beta \in L^\infty(\RR)$ satisfying \eqref{eq:Reaction} and $U_0 \in \mathcal{T}_0$ as in \eqref{eq:TraceSpace}. Then there exist $\alpha \in (0,1)$, a sequence $\vep_k \to 0^+$ and a sequence of minimizers $\{U_{\vep_k}\}_{k\in\NN}$ of \eqref{eq:Functional0} in $\UU_0$ depending only on $N$, $a$, $\|\beta\|_{L^\infty(\RR)}$ and $U_0$, such that for every open and bounded set $K \subset \RR^{N+1}\times(0,\infty)$, there exists $C > 0$ independent of $k$, such that    
\begin{equation}\label{eq:UnifHoldParIntro}
\|U_{\vep_k}\|_{C^{\alpha,\alpha/2}(K\cap\{y > 0\})} \leq C,
\end{equation}
for every $k \in \NN$.
\end{thm}
This is the first result concerning uniform H\"older bounds for minimizers of \eqref{eq:Functional0}: to the best of our knowledge, the existing literature treats exclusively uniform bounds of energy type (see the already mentioned 
\cite{AkagiStefanelli2016:art,AudritoSerraTilli2021:art,BogeleinEtAl2014:art,
BogeleinEtAl2017:art,MielkeStefanelli2011:art,SerraTilli2012:art,SerraTilli2016:art}). Theorem \ref{thm:MainIntro} has an interesting corollary, that we state after giving the definition of weak solutions to problem \eqref{eq:ParReacDiff}.

\begin{defn}\label{def:WeakSolReacDiffProb}
Let $N \geq 1$, $a \in (-1,1)$ and $\beta \in L^\infty(\RR)$ satisfying \eqref{eq:Reaction}. We say that the pair $(U,u)$ is a weak solution to \eqref{eq:ParReacDiff} if

\noindent $\bullet$ $U \in L_{loc}^2(0,\infty: H^{1,a}(\RR_+^{N+1}))$ with $\partial_t U \in L_{loc}^2(0,\infty:L^{2,a}(\RR_+^{N+1}))$.

\noindent $\bullet$ $u = U|_{y=0}$ and $U_0 = U|_{t=0}$ in the sense of traces.

\noindent $\bullet$ $U$ satisfies
\begin{equation}\label{eq:EulerParabolicIntro}
\int_0^\infty\int_{\RR_+^{N+1}} y^a (\partial_t U \eta + \nabla U \cdot\nabla\eta) \dX\rd t + \int_0^\infty\int_{\RR^N} \beta(u) \eta|_{y=0} \dx \rd t = 0,
\end{equation}
for every $\eta \in C_0^\infty(\RR^{N+1}\times(0,\infty))$.
\end{defn}
\begin{cor}\label{cor:MainIntro}
Let $N \geq 1$, $a \in (-1,1)$, $\beta \in L^\infty(\RR)$ satisfying \eqref{eq:Reaction} and $U_0 \in \mathcal{T}_0$ as in \eqref{eq:TraceSpace}. Then there exist $\alpha \in (0,1)$, a weak solution $(U,u)$ to \eqref{eq:ParReacDiff} satisfying
\[
U \in C^{\alpha,\alpha/2}_{loc}(\overline{\RR^{N+1}_+}\times(0,\infty)),
\]
a sequence $\vep_k \to 0^+$ and a sequence of minimizers $\{U_{\vep_k}\}_{k\in\NN}$ of \eqref{eq:Functional0} in $\UU_0$ such that, as $k \to + \infty$,    
\[
\begin{aligned}
&U_{\vep_k} \rightharpoonup U \quad \text{ weakly in } \UU \\
&U_{\vep_k} \to U \quad \text{ in } C_{loc}([0,\infty):L^{2,a}(\RR_+^{N+1})) \\
&U_{\vep_k} \to U \quad \text{ in } C^{\alpha,\alpha/2}_{loc}(\overline{\RR^{N+1}_+}\times(0,\infty)).
\end{aligned}
\]
\end{cor}
Some comments are in order. Our approach allows to treat both the existence and the H\"older regularity of weak solutions using the same approximating sequence, in contrast with the classical theory where the two issues are often unrelated. Indeed, the existence and H\"older regularity for weak solutions to \eqref{eq:ParReacDiff} can be proved separately using more classical methods. For the existence, we believe that the approximation scheme used in \cite[Section 2]{HyderEtAl2021:art} can be easily adapted to our framework. It is also important to notice that both methods allow to construct weak solutions with bounded $H^1$ energy (locally in time), depending on the $H^1$ energy of the initial data: this automatically excludes ``pathological'' solutions such as Jones' solution \cite{Jones1977:art}, in the case $a = 0$ and $\beta \equiv 0$.

On the other hand, the results concerning the H\"older regularity of weak solutions to parabolic weighted equations like \eqref{eq:ParReacDiff} are obtained working in the pure parabolic setting, and are based on the validity of some Harnack inequality in the spirit of Moser \cite{Moser1964:art}, see e.g. \cite{BanGarofalo2017:art,BonSimonov2019:art,ChiarenzaSerapioni1985:art,
GutierrezWheeden1991:art}. In our framework neither an \emph{elliptic} nor a \emph{parabolic} Harnack inequality for weak solutions of \eqref{eq:PerProbEps} can hold, with  constants independent of $\vep \in (0,1)$. This is due to the different nature of the \emph{elliptic} and \emph{parabolic} Harnack inequality (see \cite[Theorem 1]{Moser1961:art} and \cite[Theorem 1]{Moser1964:art}, respectively) and the drastic loss of ellipticity in the limit $\vep \to 0$ (see also the counter-example in \cite[pp. 103]{Moser1964:art} in the case $a=0$ and $\beta \equiv 0$). On the contrary, \emph{parabolic} H\"older regularity is preserved under the limit.

We end this paragraph with a few words about the reaction function $\beta$. It is worth to mention that other kind of reactions can be considered but, for simplicity, we decided to focus on the class defined in \eqref{eq:Reaction}: the positivity of $\beta$ guarantees that the functional $\FF_\vep$ is nonnegative (uniformly in $\vep$), while the fact that $\Phi(u) \leq \chi_{\{u > 0\}}$ allows us to prove the crucial level estimate  \eqref{eq:LevelEst}. The assumption $\text{supp} \beta = [0,1]$ gives some additional properties such as the weak maximum principle stated in Lemma \ref{lem:BasicPropMin}. For what concerns the uniform H\"older bounds, the only information we use is that $\beta \in L^\infty(\RR)$.

%
%
%
%
%
%%%%%%%%%%%%%%%%%%%%%%%%%%%%%%%%%%%%%%%%%%%%%%%%%%%%%%%%%%%%%%%%%%%%%%%%%%%%%%%%%%%%%%%%%%%%%%%%%%%%%%%%%%%%%%%%%%%%%%%%%%%%%%%%%%%%%%%%%%%%%%%%%%%%%
%
%
\subsection{Structure of the paper} 
The paper is organized as follows.

In Section \ref{sec:GlobalUnifEnEst} we prove the existence of minimizers of $\FF_\vep$ in $\UU_0$ (for every fixed $\vep \in (0,1)$) and we establish the main global uniform energy estimates \eqref{eq:EnBound1} and \eqref{eq:EnBound2}, which play a key role in the proof of Proposition \ref{prop:weaklimit}. This is the first main step in our analysis: we show the existence of a sequence of minimizers $U_{\vep_j}$ weakly converging to some function $U$ which is also a weak solution to \eqref{eq:ParReacDiff}.

In Section \ref{Sec:L2LinfEst} we prove a $L^{2,a} \to L^\infty$ local uniform bound for weak solutions to \eqref{eq:PerProbEpsFf} (see Proposition \ref{prop:L2LinfEstimate}). The main difficulty here is to derive a uniform energy estimate: since problem \eqref{eq:PerProbEpsFf} is elliptic for every fixed $\vep > 0$ but becomes parabolic in the limit as $\vep \to 0$, the best we can expect is to obtain a uniform energy estimate of parabolic type. We anticipate that the standard parabolic techniques do not work in this framework (see Remark \ref{rem:ParEE}) and new methods that exploit the degeneracy of the equation are used.

In Section \ref{Sec:OscDecay} we show Proposition \ref{thm:CalphaBound}: under the additional compactness assumption \eqref{eq:ConvergenceOscDecay}, there is a sequence of weak solutions to \eqref{eq:PerProbEpsFf} having locally bounded $C^{\alpha,\alpha/2}$ seminorm. As explained in Remark \ref{rem:CompactnessPar}, this compactness assumption is required in order to prove a parabolic version of the so-called ``De Giorgi isoperimetric lemma'' (cf. Lemma \ref{lem:IsoperimetricIneqElliptic}).

In Section \ref{Sec:FinalProofs} we show Theorem \ref{thm:MainIntro} and Corollary \ref{cor:MainIntro}: the proofs are easy consequences of Proposition \ref{prop:weaklimit}, Proposition \ref{thm:CalphaBound} and a standard covering argument.

Finally, in  the appendices (Appendix \ref{App:WeightedSpaces}, \ref{App:TechnicalLemma} and \ref{App:Notations}) we recall some technical tools and results we exploit through the paper, and the full list of notations.

%%%%%%%%%%%%%%%%%%%%%%%%%%%%%%%%%%%%%%%%%%%%%%%%%%%%%%%%%%%%%%%%%%%%%%%%%%%%%%%%%%%%%%%%%%%%%%%%%%%%%%%%%%%%%

%%%%%%%%%%%%%%%%%%%%%%%%%%%%%%%%%%%%%%%%%%%%%%%%%%%%%%%%%%%%%%%%%%%%%%%%%%%%%%%%%%%%%%%%%%%%%%%%%%%%%%%%%%%%%
%
%
%
%
\section{Global uniform energy estimates}\label{sec:GlobalUnifEnEst}
This section is devoted to the proof of the following proposition.
\begin{prop}\label{prop:weaklimit} 
Let $\{U_\vep\}_{\vep \in (0,1)} \in \UU_0$ be a family of minimizers of $\FF_\vep$. Then there exist $U \in \UU_0$ and a sequence $\vep_j \to 0$ such that if $u = U|_{y=0}$, then 
\begin{equation}\label{eq:WeakLimitProp}
\begin{aligned}
&U_{\vep_j} \rightharpoonup U \quad \text{ weakly in } \UU \\
&U_{\vep_j} \to U \quad \text{ in } C_{loc}([0,\infty):L^{2,a}(\RR^{N+1})) \\
&u_{\vep_j} \to u \quad \;\text{ in } L^2_{loc}(\RR^N\times(0,\infty)),
\end{aligned}
\end{equation}
and, furthermore, the pair $(U,u)$ satisfies
\begin{equation}\label{eq:EulerParabolic}
\int_{\mathbb{Q}_\infty} |y|^a (\partial_t U \eta + \nabla U \cdot\nabla\eta) \dX\rd t + \int_{Q_\infty} \beta(u) \eta|_{y=0} \dx \rd t = 0,
\end{equation}
for every $\eta \in C_0^\infty(\mathbb{Q}_\infty)$. In particular, $(U,u)$ is a weak solution to \eqref{eq:ParReacDiff}.
\end{prop}
Before addressing to the proof of the above statement, we show some basic properties of minimizers of the functional $\FF_\vep$.
%
%%%%%%%%%%%%%%%%%%%%%%%%%%%%%%%%%%%%%%%%%%%%%%%%%%%%%%%%%%%%%%%%%%%%%%%%%%%%%%%%%%%%%%%%%%%%%%%%%%%%%%%%%%%%%%%%%%%%%%%%%%%%%%%%%%%%%%%%%%%%%%%%%%%%%
%
\subsection{Existence and basic properties of minimizers}
We begin with the Euler-Lagrange equations.
\begin{lem}\label{Lem:EulerEqMinFF}
Fix $\vep \in (0,1)$ and let $U_\vep \in \UU_0$ be a minimizer of $\FF_\vep$. Then
\begin{equation}\label{eq:EulerEqMinFF}
\varepsilon \int_{\mathbb{Q}_\infty} |y|^a \partial_t U_\vep \partial_t \eta \dX\rd t + \int_{\mathbb{Q}_\infty} |y|^a (\partial_t U_\vep \eta + \nabla U_\vep \cdot\nabla\eta) \dX\rd t + \int_{Q_\infty} \beta(u_\vep) \eta|_{y=0} \dx \rd t = 0,
\end{equation}
for every $\eta \in C_0^\infty(\mathbb{Q}_\infty)$. 
\end{lem}
\begin{proof} Fix $\varepsilon > 0$, let $\varphi \in C_0^\infty(\mathbb{Q}_\infty)$, $h \not= 0$, and assume that $U:= U_\vep \in \UU_0$ is a minimizer of $\FF_\vep$. It is not difficult to compute
\[
\begin{aligned}
\frac{\FF_\vep(U + h\varphi) - \FF_\vep(U)}{h} &= 2\int_0^\infty \frac{e^{-t/\varepsilon}}{\varepsilon} \int_{\mathbb{R}^{N+1}} |y|^a \left( \varepsilon \partial_t U \partial_t\varphi + \nabla U \cdot\nabla\varphi \right)\rd X \rd t \\
& \quad + \int_0^\infty \frac{e^{-t/\varepsilon}}{\varepsilon} \int_{\mathbb{R}^N} \frac{\Phi(u+h\varphi|_{y=0}) - \Phi(u)}{h} \dx \rd t \\
& \quad + h \int_0^\infty \frac{e^{-t/\varepsilon}}{\varepsilon} \int_{\mathbb{R}^{N+1}} |y|^a \left( \varepsilon |\partial_t\varphi|^2 + |\nabla\varphi|^2 \right)\rd X \rd t.
\end{aligned}
\]
Notice that, since $\varphi \in C_0^\infty(\mathbb{Q}_\infty)$, the last integral converges to zero as $h \to 0$, while
\[
\frac{\Phi(u+h\varphi|_{y=0}) - \Phi(u)}{h} \to \Phi'(u)\varphi|_{y=0} = 2\beta(u)\varphi|_{y=0} \quad \text{ a.e. in } Q_\infty,
\]
as $h \to 0$. Consequently, using that $\beta$ is bounded, $\varphi|_{y=0}$ is compactly supported and the minimality of $U$, we can pass to the limit as $h \to 0$, to deduce
\[
\int_0^\infty e^{-t/\varepsilon} \bigg\{ \int_{\mathbb{R}^{N+1}} |y|^a \left( \varepsilon \partial_t U \partial_t\varphi + \nabla U \cdot \nabla\varphi \right)\rd X + \int_{\mathbb{R}^N} \beta(u)\varphi|_{y=0}\dx  \bigg\} \rd t = 0.
\]
Now, $\eta \in C_0^\infty(\mathbb{Q}_\infty)$ and take $\varphi = e^{t/\varepsilon} \eta$. Noticing  $\partial_t\varphi = e^{t/\varepsilon}\left(\tfrac{1}{\varepsilon} \eta + \partial_t\eta \right)$ and rearranging terms, \eqref{eq:EulerEqMinFF} follows.
\end{proof}
Now, we show that for every $\vep \in (0,1)$ the functional $\FF_\vep$ has a minimizer in $\UU_0$. To simplify the notations, we introduce the functional
\begin{equation}\label{eq:ResFunctional}
\JJ_\vep(V) := \int_0^\infty e^{- t} \bigg\{ \int_{\mathbb{R}^{N+1}} |y|^a \left( |\partial_t V|^2 + \vep|\nabla V|^2 \right)\rd X + \vep \int_{\mathbb{R}^N \times\{0\}} \Phi(v) \dx \bigg\}\dt ,
\end{equation}
defined for $V \in \UU$, where $v = V|_{y=0}$. This functional is related to \eqref{eq:Functional} through the following relations
\begin{equation}\label{eq:RelFuncFJ}
\FF_\vep(U) = \frac{1}{\varepsilon} \JJ_\vep(V), \qquad V(X,t) = U(X,\varepsilon t).
\end{equation}
Since $V_0 := V|_{t=0} = U|_{t=0}$ (and so $v_0 := V_0|_{y=0} = u_0$) and $\UU_0$ is convex and invariant under time transformations, the minimization of $\JJ_\vep$ on $\UU_0$ is equivalent to the minimization of $\FF_\vep$ on the same space. In other words, $U \in \UU_0$ is a minimizer of $\FF_\vep$ if and only if $V \in \UU_0$ is a minimizer of $\JJ_\vep$.
\begin{lem}\label{lem:ExistenceMin}
For every $\varepsilon \in (0,1)$, the functional $\JJ_\vep$ defined in \eqref{eq:ResFunctional} has a minimizer in $\UU_0$.

Further, there exists a constant $C > 0$ depending only on $N$, $a$ and $V_0$ such that for every family $\{V_\vep\}_{\vep \in (0,1)} \in \UU_0$ of minimizers of $\JJ_\vep$, we have
\begin{equation}\label{eq:LevelEst}
\JJ_\vep(V_\vep) \leq C \varepsilon.
\end{equation}
\end{lem}
\begin{proof}
Let $V = V_\vep$. First, we have $\JJ_\vep \not \equiv +\infty$ on $\UU_0$. Indeed, $V_0 \in \UU_0$ by definition and, further,
\[
\begin{aligned}
\JJ_\vep(V_0) & = \int_0^\infty e^{-t} \bigg\{ \int_{\mathbb{R}^{N+1}} |y|^a \left(|\partial_t V_0|^2 + \varepsilon |\nabla V_0|^2 \right)\rd X +  \varepsilon \int_{\mathbb{R}^N} \Phi(v_0) \dx \bigg\}\dt \\
& \leq \varepsilon \int_{\mathbb{R}^{N+1}} |y|^a |\nabla V_0|^2 \rd X + \varepsilon \int_{\RR^N} \chi_{\{v_0>0\}} \dx \leq \varepsilon \| V_0\|_{H^{1,a}(\RR^{N+1})}^2  + \varepsilon \mathcal{L}^N(\{v_0 > 0\}) \leq C\varepsilon,
\end{aligned}
\]
where we have used H\"older's inequality and assumption \eqref{eq:Ass2Data}. From the above inequality, \eqref{eq:LevelEst} follows too: if $V \in \UU_0$ is a minimizer, we have $\JJ_\vep(V) \leq \JJ_\vep(V_0) \leq C\varepsilon$ and $C$ depends only on $N$, $a$ and $V_0$.

We are left to prove the existence of a minimizer. Since $V_0 \in \UU_0$ and $\JJ_\vep(V_0) < + \infty$, there exists a minimizing sequence $\{V_j\}_{j\in\mathbb{N}} \subset \UU_0$ :
\begin{equation}\label{eq:MinSeqCond}
\lim_{j \to +\infty} \JJ_\vep(V_j) = \inf_{\tilde{V}\in\UU_0} \JJ_\vep(\tilde{V}) \in [0,+\infty).
\end{equation}
In particular, for every fixed $R > 0$, we have 
\[
\int_{\QQ_R^+} |y|^a \left(|\partial_t V_j|^2 + \varepsilon |\nabla V_j|^2 \right)\rd X \dt \leq C_R,
\]
for some $C_R > 0$ independent of $j$ and so, since $V_j|_{t=0} = V_0|_{t=0}$ for every $j$, $\{V_j\}_{j\in\mathbb{N}}$ is uniformly bounded in $H^{1,a}(\QQ_R^+)$. By the compactness of the inclusion $L^{2,a}(\QQ_R^+) \hookrightarrow H^{1,a}(\QQ_R^+)$, there exists $V \in H^{1,a}(\QQ_R^+)$ such that $V_j \rightharpoonup V$ weakly in $H^{1,a}(\QQ_R^+)$ and $V_j \to V$ a.e. in $\QQ_R^+$ and in $L^{2,a}(\QQ_R^+)$, up to passing to a suitable subsequence, still denoted with $V_j$.

Similar, setting $v_j := V_j|_{y=0}$, we have that $\{v_j\}_{j\in\mathbb{N}}$ is uniformly bounded in $H^{\frac{1-a}{2}}(Q_R^+)$ by the trace theorem (see for instance \cite{Nekvinda1993:art}) and so, up to a subsequence, $v_j \to v$ a.e. in $Q_R^+$ and in $L^2(Q_R^+)$, where $v := V|_{y=0}$.

Now, since $\mathbb{Q}_\infty = \cup_{R>0} \QQ_R^+$ and $Q_\infty = \cup_{R>0} Q_R^+$, a standard diagonal argument shows that
\[
\begin{aligned}
&V_j \rightharpoonup V \quad \text{weakly in } \UU \\
&V_j \to V \quad \text{a.e. in } \mathbb{Q}_\infty \text{ and in } L^{2,a}(\QQ_\infty), \\
&v_j \to v  \quad \text{a.e. in } Q_\infty \text{ and in } L^2(Q_\infty),
\end{aligned}
\]
up to passing to an additional subsequence. Notice that since $\UU_0$ is closed and convex, we have $V \in \UU_0$. Further, by continuity, we have $\Phi(v_j) \to \Phi(v)$ a.e. in $Q$ and so, by lower semicontinuity and Fatou's lemma, it follows
\[
\JJ_\vep(V) \leq \lim_{j \to +\infty} \JJ_\vep(V_j) = \inf_{\widetilde{V} \in \UU_0} \JJ_\vep(\widetilde{V}),
\]
i.e., $V$ is a minimizer of $\JJ_\vep$. 
\end{proof}
\begin{lem}\label{lem:BasicPropMin}
Let $\vep \in (0,1)$ be fixed. Then:

\smallskip

\noindent $\bullet$ If $V_0 \geq 0$ a.e. and $V_\vep \in \UU_0$ is a minimizer of $\JJ_\vep$, then $V_\vep \geq 0$ a.e. (and $V_\vep\not\equiv 0$).

\smallskip

\noindent $\bullet$ If $V_0 \leq 1$ a.e. and $V_\vep \in \UU_0$ is a minimizer of $\JJ_\vep$, then $V_\vep \leq 1$ a.e.

\smallskip

\noindent $\bullet$ If $V_0$ is even w.r.t. $y$, then there exists a minimizer of $\JJ_\vep$ in $\UU_0$ which is even w.r.t. $y$.
\end{lem}
\begin{proof}
Let $V := V_\vep \in \UU_0$ be a minimizer of $\JJ_\vep$ and assume $V_0 \geq 0$ a.e.. Then $V_+$ is an admissible competitor, with $\JJ_\vep(V_+) < \JJ_\vep(V)$, unless $V \geq 0$ a.e.

Similar if $V_0 \leq 1$ a.e., $W := \min\{V,1\}$ is an admissible competitor and, since $\Phi(v) = 1$ for $v \geq 1$ in view of \eqref{eq:Reaction}, we have $\JJ_\vep(W) < \JJ_\vep(V)$, unless $V \leq 1$ a.e.

Finally, if $V_0$ is even w.r.t. $y$ and $V \in \UU_0$ is a minimizer of $\JJ_\vep$, then 
\[
V_e(x,y,t) :=
\begin{cases}
V(x,y,t) \quad &\text{ if } y \geq 0 \\
V(x,-y,t) \quad &\text{ if } y < 0
\end{cases}
\]
is an admissible competitor, with $\JJ_\vep(V_e) = \JJ_\vep(V)$. 
\end{proof}
\begin{rem}\label{rem:EvenPropMin} The last point of the above statement and Lemma \ref{lem:ExistenceMin} tell us that if the initial data is even w.r.t. $y$, then we may assume that $\FF_\vep$ has a minimizer $U_\vep$ which is even w.r.t. $y$. Such minimizer satisfy
\[
\FF_\vep(U_\vep) = 2 \EE_\vep(U_\vep|_{y > 0}),
\]
where $\EE_\vep$ is defined in \eqref{eq:Functional0} and $U_\vep|_{y > 0}$ is the restriction of $U_\vep$ to $\RR^{N+1}_+\times(0,\infty)$. It thus turns out that minimizing $\FF_\vep$ in $\UU_0$ (with $U_0$ even w.r.t. $y$) is equivalent to minimizing $\EE_\vep$ in the space $\UU_0^+ := \{U|_{y > 0}: U \in \UU_0\}$.
\end{rem}
%
%
%
%
%%%%%%%%%%%%%%%%%%%%%%%%%%%%%%%%%%%%%%%%%%%%%%%%%%%%%%%%%%%%%%%%%%%%%%%%%%%%%%%%%%%%%%%%%%%%%%%%%%%%%%%%%%%%%%%%%%%%%%%%%%%%%%%%%%%%%%%%%%%%%%%%%%%%%
%
\subsection{Proof of Proposition \ref{prop:weaklimit}}

Proposition \ref{prop:weaklimit} will be obtained as a consequence of the following energy estimates.
\begin{prop}\label{thm:UnifEnergyEst} (Global uniform energy estimates)
There exists $C > 0$ depending only on $N$, $a$ and $U_0$ such that for every family $\{U_\vep\}_{\vep \in (0,1)} \in \UU_0$ of minimizers of $\FF_\vep$, we have
\begin{equation}\label{eq:EnBound1}
\int_0^\infty \int_{\mathbb{R}^{N+1}} |y|^a |\partial_t U_\vep|^2\dX \rd\tau \leq C,
\end{equation}
and, for every $R \geq \varepsilon$,
\begin{equation}\label{eq:EnBound2}
\int_0^R \int_{\mathbb{R}^{N+1}} |y|^a |\nabla U_\vep|^2\dX\rd\tau + \int_0^R \int_{\mathbb{R}^N\times\{0\}} \Phi(u_\vep)\dx\rd\tau \leq CR.
\end{equation}
\end{prop}
The above statement is the key result of this section and will be proved by combining Lemma \ref{Lemma:DerivEn} and Corollary \ref{cor:GlobalEnEstV} that we show below. As in the above subsection, we consider a minimizer $V$ of $\JJ_\vep$ and we write
\begin{equation}\label{eq:ShortJFuncInner}
\JJ_\vep(V) = \int_0^\infty e^{-t} \left[ \I(t) + \R(t) \right] \rd t,
\end{equation}
where
\begin{equation}\label{eq:DefIandR}
\I(t) := \int_{\mathbb{R}^{N+1}} |y|^a |\partial_t V|^2\dX, \qquad \R(t) := \varepsilon \int_{\mathbb{R}^{N+1}} |y|^a |\nabla V|^2\dX + \varepsilon \int_{\mathbb{R}^N} \Phi(v)\dx.
\end{equation}
Notice that, since $V$ is a minimizer, we have $\I,\R \in L^1_{loc}(\mathbb{R}_+)$ and $t \to e^{-t} \left[ \I(t) + \R(t) \right] \in L^1(\mathbb{R}_+)$. Further, we introduce the function
\begin{equation}\label{eq:Energy}
\E(t) := e^t \int_t^\infty e^{-\tau} \left[ \I(\tau) + \R(\tau) \right] \rd \tau,
\end{equation}
which belongs to $W^{1,1}_{loc}(\mathbb{R}_+)\cap C(\overline{\mathbb{R}_+})$, and satisfies $E(0) = \JJ_\vep(V)$ and
\begin{equation}\label{eq:DerivEn1}
\E' = \E - \I - \R \quad \text{ in } \mathcal{D}'(\mathbb{R}_+).
\end{equation}
The main idea of the following lemma is to find a different expression for the derivative of the function $\E$ defined in \eqref{eq:Energy}. The new formulation for $\E'$ is crucial to prove our main estimates.
\begin{lem}\label{Lemma:DerivEn}
Let $V \in \UU_0$ be a minimizer of $\JJ_\vep$. Then
\begin{equation}\label{eq:DerivEn2}
\E' = - 2\I \quad \text{ in } \mathcal{D}'(\mathbb{R}_+).
\end{equation}
\end{lem}
\begin{proof} We follow the proof of \cite[Proposition 3.1]{SerraTilli2012:art} (see also \cite[Lemma 4.5]{BogeleinEtAl2014:art} and \cite[Lemma 4.2]{AudritoSerraTilli2021:art}).

Let $V \in \UU_0$ be a minimizer of $\JJ_\vep$. Fix $\eta \in C_0^\infty(0,\infty)$, consider $\zeta(t) := \int_0^t\eta(\tau)\rd\tau$, and, given $\lambda \in \mathbb{R}$, define
\begin{equation}\label{eq:DefPhi}
\varphi(t) := t - \lambda\zeta(t), \quad t \geq 0.
\end{equation}
It is easily seen that, if $|\lambda| \leq \lambda_0$ for some $\lambda_0 > 0$ small enough, then $\varphi$ is strictly increasing with $\varphi(0) = 0$. In particular, the inverse $\psi = \varphi^{-1}$ exists, it is smooth and, by \eqref{eq:DefPhi}, satisfies
\begin{equation}\label{eq:DefPsi}
\psi(\tau) = \tau + \lambda \zeta(\psi(\tau)).
\end{equation}
The key idea of the proof is to use the function $\varphi$ to construct a competitor $W$. It is obtained as an inner variation of $V$:
\[
W(X,t) := V(X,\varphi(t)).
\]
Since $\varphi(0) = 0$, we have $W = V$ when $t=0$ and so $W \in \UU_0$. Further, by \eqref{eq:DefPhi}, $W = V$ when $\lambda = 0$ (by sake of simplicity, the dependence on $\lambda$ is omitted in the notations for $\varphi$, $\psi$ and $W$).

Now, from the formulation of $\JJ_\vep$ introduced in \eqref{eq:ShortJFuncInner} and the change of variable $t = \psi(\tau)$, we have
\[
\begin{aligned}
\JJ_\vep(W) &= \int_0^\infty e^{-t} \bigg\{ \int_{\mathbb{R}^{N+1}_+} y^a \left(|\partial_t W|^2 + \varepsilon |\nabla W|^2 \right)\rd X + \varepsilon \int_{\mathbb{R}^N} \Phi(w)\dx \bigg\}\dt \\
&= \int_0^\infty e^{-t} \left[ \varphi'(t)^2 \, \I(\varphi(t)) + \R(\varphi(t)) \right] \rd t = \int_0^\infty \psi'(\tau) e^{-\psi(\tau)} \left[ \varphi'(\psi(\tau))^2 \,\I(\tau) + \R(\tau) \right] \rd \tau.
\end{aligned}
\]
In view of \eqref{eq:DefPhi} and \eqref{eq:DefPsi}, $\varphi',\psi' \in L^\infty(\mathbb{R}^+)$ and $e^{-\psi(\tau)} \leq e^{\lambda \|\zeta\|_{L^\infty(\mathbb{R}_+)}} e^{-\tau}$, and thus $\JJ_\vep(W) < +\infty$. In particular, we deduce that, for any small $\lambda$ ($|\lambda| \leq \lambda_0$), $W \in \UU_0$ is an admissible competitor. Actually, recalling that $W = V$ when $\lambda = 0$ and $V$ is a minimizer, it must be
\begin{equation}\label{eq:GatDerJInner}
\lim_{\lambda \to 0^+} \frac{\JJ_\vep(W) - \JJ_\vep(V)}{\lambda} = 0.
\end{equation}
Proceeding exactly as in \cite[Lemma 4.2]{AudritoSerraTilli2021:art}, we compute
\begin{equation}\label{eq:CompInner}
\frac{\partial}{\partial\lambda} \left( \psi'(\tau) e^{-\psi(\tau)} \right) \Big|_{\lambda = 0} =  \zeta'(\tau)e^{-\tau} - \zeta(\tau)e^{-\tau}, \qquad \frac{\partial}{\partial\lambda} \left|\varphi'(\psi(\tau))\right|^2 \Big|_{\lambda = 0} = -2\zeta'(\tau).
\end{equation}
Consequently, recalling that $t \to e^{-t} \left\{ \I(t) + \R(t) \right\} \in L^1(\mathbb{R}_+)$ and using the dominated convergence theorem, we can pass to the limit in \eqref{eq:GatDerJInner} and, making use of \eqref{eq:CompInner}, we can write \eqref{eq:GatDerJInner} explicitly:
\begin{equation}\label{eq:MinCondInner}
\int_0^\infty \big[ \zeta'(\tau)e^{-\tau} - \zeta(\tau)e^{-\tau} \big] \big[ \I(\tau) + \R(\tau) \big] \rd\tau - 2 \int_0^\infty e^{-\tau}\zeta'(\tau) \, \I(\tau)  \dtau = 0,
\end{equation}
where we have used $\psi(\tau) = \tau$, and $\varphi' = \psi' = 1$ when $\lambda = 0$ (see \eqref{eq:DefPhi} and \eqref{eq:DefPsi}).

Now, we show how equation \eqref{eq:MinCondInner} leads to \eqref{eq:DerivEn2}. Recalling that $\zeta' = \eta$ and writing \eqref{eq:DerivEn1} with test function $\zeta'(\tau)e^{-\tau}$, we easily obtain
\begin{equation}\label{eq:AuxForInner1}
\begin{aligned}
\int_0^\infty \zeta'(\tau)e^{-\tau} \big[ \I(\tau) + \R(\tau) \big] \rd\tau &= \int_0^\infty  \E(\tau) \big[ \zeta'(\tau) e^{-\tau} + \big( \zeta'(\tau)e^{-\tau} \big)' \big] \rd\tau \\
&= \int_0^\infty \zeta'(\tau) e^{-\tau} \E(\tau) \dtau + \int_0^\infty \E(\tau) \big( \eta(\tau)e^{-\tau} \big)' \rd\tau.
\end{aligned}
\end{equation}
Further, using the definition of $\E$ given in \eqref{eq:Energy} and integrating by parts, it follows
\begin{equation}\label{eq:AuxForInner2}
\begin{aligned}
\int_0^\infty \zeta'(\tau) e^{-\tau} \E(\tau) \dtau &= \int_0^\infty \zeta'(\tau) \int_\tau^\infty e^{-\vartheta} \big[ \I(\vartheta) + \R(\vartheta) \big] \rd\vartheta \rd \tau \\
&= \int_0^\infty \zeta(\tau) e^{-\tau} \big[ \I(\tau) + \R(\tau) \big] \rd \tau.
\end{aligned}
\end{equation}
The ``boundary terms'' in the integration by parts disappear since $\zeta(0) = 0$ and $e^{-t}\E(t) \to 0$ as $t \to +\infty$. Finally, plugging \eqref{eq:AuxForInner1} and \eqref{eq:AuxForInner2} into \eqref{eq:MinCondInner}, it follows
\[
\int_0^\infty \E(\tau) \big( e^{-\tau} \eta(\tau) \big)' \rd\tau = 2 \int_0^\infty e^{-\tau} \eta(\tau) \, \I(\tau)  \dtau,
\]
and, since $\eta \in C_0^\infty(\mathbb{R}^+)$ is arbitrary, \eqref{eq:DerivEn2} is proved.
\end{proof}

\begin{cor}\label{cor:GlobalEnEstV}
Let $V \in \UU_0$ be a minimizer of $\JJ_\vep$. Then
\begin{equation}\label{eq:BoundI}
\int_0^\infty \int_{\mathbb{R}^{N+1}} |y|^a |\partial_\tau V|^2\dX \rd\tau \leq C\varepsilon,
\end{equation}
and, for every $t \geq 0$,
\begin{equation}\label{eq:BoundR}
\int_t^{t+1}\int_{\mathbb{R}^{N+1}} |y|^a |\nabla V|^2\dX\rd\tau + \int_t^{t+1} \int_{\mathbb{R}^N\times\{0\}} \Phi(v)\dx\rd\tau \leq C,
\end{equation}
for some constant $C > 0$ depending only on $N$, $a$ and $U_0$.
\end{cor}
\begin{proof} First, recalling that $\E \in C(\overline{\mathbb{R}_+})$ and $\I \geq 0$, a direct integration of \eqref{eq:DerivEn2} shows that $\E(t) \leq \E(0)$ for all $t \geq 0$. Moreover, since $\E(0) = \JJ_\vep(V)$ and \eqref{eq:LevelEst}, we obtain
\begin{equation}\label{eq:EnergyBoundInner}
\E(t) \leq \JJ_\vep(V) \leq C\varepsilon,
\end{equation}
for all $t \geq 0$, where $C > 0$ depends only on $N$, $a$ and $V_0$.

From \eqref{eq:DerivEn2}, \eqref{eq:EnergyBoundInner} and that $\E \geq 0$, we deduce
\[
2\int_0^t \I(\tau)\dtau = \E(0) - \E(\tau) \leq \E(0) \leq C\varepsilon,
\]
for all $t \geq 0$. Consequently, \eqref{eq:BoundI} follows by passing to the limit as $t \to +\infty$ and the first definition in \eqref{eq:DefIandR}.

To prove the second part of the statement, we fix $t \geq 0$ and we notice that \eqref{eq:EnergyBoundInner} implies
\[
\int_t^{t+1} \R(\tau) \dtau \leq e^{t+1} \int_t^{t+1} e^{-\tau} \R(\tau) \dtau \leq e \E(t) \leq e C\varepsilon.
\]
The thesis follows by the second definition in \eqref{eq:DefIandR}.
\end{proof}
\begin{proof}[Proof of Proposition \ref{thm:UnifEnergyEst}] Let $U$ be a minimizer of $\FF_\vep$ in $\UU_0$ and define $V(X,t) := U(X,\varepsilon t)$. Then, $V$ is a minimizer of $\JJ_\vep$ in $U_0$.

As a first consequence, we immediately see that \eqref{eq:EnBound1} follows changing variable ($t = \varepsilon \tau$) in \eqref{eq:BoundI}. Similar, the same change of variable in \eqref{eq:BoundR} yields
\[
\int_{\varepsilon t}^{\varepsilon t + \varepsilon}\int_{\mathbb{R}^{N+1}} |y|^a |\nabla V|^2\dX\rd\tau + \int_{\varepsilon t}^{\varepsilon t + \varepsilon} \int_{\mathbb{R}^N} \Phi(v)\dx\rd\tau \leq C\varepsilon,
\]
for all $t \geq 0$. Now, let $R \geq \varepsilon$ and define $k = \lceil R/\varepsilon \rceil$. In view of the arbitrariness of $t$, we can apply the above estimate for $t = j$ and sum over $j = 0,\dots,k-1$ to obtain
\[
\int_0^{k\varepsilon}\int_{\mathbb{R}^{N+1}} |y|^a |\nabla V|^2\dX\rd\tau + \int_0^{k\varepsilon} \int_{\mathbb{R}^N} \Phi(v)\dx\rd\tau \leq Ck\varepsilon.
\]
Since $\varepsilon \leq R \leq k\varepsilon$, \eqref{eq:EnBound2} follows.
\end{proof}
\begin{proof}[Proof of Proposition \ref{prop:weaklimit}] In view of \eqref{eq:EnBound1} and \eqref{eq:EnBound2}, $\{U_\vep\}_{\vep \in (0,1)}$ is equibounded in $H^{1,a}(\mathbb{Q}_R^+)$ for every fixed $R > 0$. Consequently, the usual diagonal procedure shows the existence of a sequence $\varepsilon_j \to 0$ and $U \in \UU_0$ such that the first and the third limit in \eqref{eq:WeakLimitProp} are satisfied (here we may also use Sobolev embedding and trace theorems as in Lemma \ref{lem:ExistenceMin}). The second limit in \eqref{eq:WeakLimitProp} follows by \cite[Corollary 8]{Simon1987:art}, up to passing to another subsequence, since for every $R >0$, $\{U_{\vep_j}\}_{j\in \NN}$ is uniformly bounded in $L^2(0,R^2:H^{1,a}(\BB_R))$ and $\{\partial_t U_{\vep_j}\}_{j\in \NN}$ is uniformly bounded in $L^2(0,R^2:L^{2,a}(\BB_R))$. 

Finally, \eqref{eq:EulerParabolic} follows by passing into the limit as $j \to +\infty$ into \eqref{eq:EulerEqMinFF} (with $\vep = \vep_j$) and using \eqref{eq:WeakLimitProp}. 
\end{proof}

%
%
%%%%%%%%%%%%%%%%%%%%%%%%%%%%%%%%%%%%%%%%%%%%%%%%%%%%%%%%%%%%%%%%%%%%%%%%%%%%%%%%%%%%%%%%%%%%%%%%%%%%%%%%%%%%%%
%
%%%%%%%%%%%%%%%%%%%%%%%%%%%%%%%%%%%%%%%%%%%%%%%%%%%%%%%%%%%%%%%%%%%%%%%%%%%%%%%%%%%%%%%%%%%%%%%%%%%%%%%%%%%%%%
%
%

%
\section{Uniform \texorpdfstring{$L^{2,a} \to L^\infty$}{} estimates}\label{Sec:L2LinfEst}
This section is devoted to the proof of some \emph{local} and \emph{uniform} $L^{2,a} \to L^\infty$ estimates for weak solutions to the linear problem \eqref{eq:PerProbEpsFf} (the notion of weak solution is introduced in Definition \ref{def:WeakSolFf} below). 
\begin{prop}\label{prop:L2LinfEstimate}(Uniform $L^{2,a} \to L^\infty$ estimate)
Let $N \geq 1$, $a \in (-1,1)$ and $(p,q)$ satisfying \eqref{eq:AssumptionsExponents}. Then there exists a constant $C > 0$ depending only on $N$, $a$ and $q$ such that every family $\{U_\varepsilon\}_{\varepsilon \in (0,1)}$ of weak solutions to problem \eqref{eq:PerProbEpsFf} in $\QQ_1$ in the sense of Definition \ref{def:WeakSolFf} satisfy
\begin{equation}\label{eq:L2LinfEstimate}
\begin{aligned}
\| U_\vep \|_{L^\infty(\QQ_{1/2})} \leq C \big( \| U_\vep \|_{L^{2,a}(\QQ_1)} + \|F_\vep\|_{L^{p,a}(\QQ_1)} + \|f_\vep\|_{L_\infty^q(Q_1)} \big),
\end{aligned}
\end{equation}
for every $\vep \in (0,1)$.
\end{prop}
We divide the proof of Proposition \ref{prop:L2LinfEstimate} in two main steps: we establish an energy estimate for solutions and nonnegative subsolutions (Lemma \ref{lem:EnergyEstimate}) and then we exploit it to prove a ``no-spikes'' estimate (see Lemma \ref{lem:NoSpikesEstimate}).

Before moving forward, we fix some important notations and give the definition of weak solution to problem \eqref{eq:PerProbEpsFf}. Let $N \geq 1$ and $a \in (-1,1)$. We will always consider exponents $p,q \in \RR$ satisfying
\begin{equation}\label{eq:AssumptionsExponents}
p > \bar{p} := \max\Big\{ \frac{N+3+a}{2},2 \Big\} \qquad q > \frac{N}{1-a}.
\end{equation}
We anticipate that the assumption $p > 2$ is needed only when $N=1$ and $a \in (-1,0]$ (for all other values of $N$ and $a$ we have $\bar{p} = (N+3+a)/2$): this is due to the fact that this range of parameters is critical for the Sobolev inequality (cf. Theorem \ref{thm:SobEmbedLocalEll} and Theorem \ref{thm:SobEmbedLocalPar}). 

\medskip

Let us proceed with the notion of weak solutions to \eqref{eq:PerProbEpsFf}. It is related to Definition \ref{def:WeakSolReacDiffProb}, for problem \eqref{eq:ParReacDiff}: let $R > 0$, $\vep \in (0,1)$, $(p,q)$ satisfying \eqref{eq:AssumptionsExponents}, $F_\vep \in L^{p,a}(\BB_R^+\times(-R^2,R^2))$ and $f_\vep \in L_\infty^q(Q_R)$
\footnote{If $Q := B\times I$, $\|f\|_{L_\infty^q(Q)} := \esssup_{t\in I} \|f(t)\|_{L^q(B)}$, and $L_\infty^q(Q)$ is the closure of $C^\infty(Q)$ w.r.t. $\|\cdot\|_{L_\infty^q(Q)}$.}. 
We say that $W_\vep$ is a weak solution to \eqref{eq:PerProbEpsFf} in $\BB_R^+\times(-R^2,R^2)$ if

\noindent $\bullet$ $W_\vep \in L^2(-R^2,R^2: H^{1,a}(\BB_R^+))$ with $\partial_t W_\vep \in L^2(-R^2,R^2:L^{2,a}(\BB_R^+))$.

\noindent $\bullet$ $W_\vep$ satisfies
\[
\int_{-R^2}^{R^2}\int_{\BB_R^+} y^a ( \varepsilon \partial_t W_\vep \partial_t \eta + \partial_t W_\vep \eta + \nabla W_\vep \cdot\nabla\eta - F_\vep \eta) \dX\rd t - \int_{-R^2}^{R^2}\int_{B_R} f_\vep \eta|_{y=0} \dx \rd t = 0,
\]
for every $\eta \in C_0^\infty(\QQ_R)$.

Even in the setting of problem \eqref{eq:PerProbEpsFf}, it is convenient to simplify the notation as follows. If $W_\vep$ is a weak solution, we notice that the even extension
\[
U_\vep(x,y,t) :=
\begin{cases}
W_\vep(x,y,t) \quad &\text{ if } y \geq 0 \\
W_\vep(x,-y,t) \quad &\text{ otherwise},
\end{cases}
\]
satisfies
\[
\int_{\mathbb{Q}_R} |y|^a ( \varepsilon \partial_t U_\vep \partial_t \eta + \partial_t U_\vep \eta + \nabla U_\vep \cdot\nabla\eta - \tilde{F}_\vep \eta) \dX\rd t - \int_{Q_R} \tilde{f}_\vep \eta|_{y=0} \dx \rd t = 0,
\]
for every $\eta \in C_0^\infty(\QQ_R)$, where $\tilde{F}_\vep$ is the even extension of $F_\vep$ and $\tilde{f}_\vep := 2 f_\vep$. For this reason, we will always work with  weak solutions defined in the whole cylinder $\QQ_R$ and, to recover the information on $W_\vep$, we restrict them to the upper-half cylinder $\QQ_R\cap\{y>0\}$.
\begin{defn}\label{def:WeakSolFf} Let $N \geq 1$, $a \in (-1,1)$, $R > 0$, $\vep \in (0,1)$, $(p,q)$ satisfying \eqref{eq:AssumptionsExponents}, $F_\vep \in L^{p,a}(\QQ_R)$ and $f_\vep \in L_\infty^q(Q_R)$. We say that $U_\vep$ is a weak subsolution (supersolution) of \eqref{eq:PerProbEpsFf} in $\QQ_R$ if

\noindent $\bullet$ $U_\vep \in L^2(-R^2,R^2; H^{1,a}(\mathbb{B}_R))$ with $\partial_t U_\varepsilon \in L^2(-R^2,R^2:L^{2,a}(\mathbb{B}_R))$.

\noindent $\bullet$ $U_\vep$ satisfies the differential inequality
\[
\int_{\mathbb{Q}_R} |y|^a ( \varepsilon \partial_t U_\vep \partial_t \eta + \partial_t U_\vep \eta + \nabla U_\vep \cdot\nabla\eta - \tilde{F}_\vep \eta) \dX\rd t - \int_{Q_R} \tilde{f}_\vep \eta|_{y=0} \dx \rd t \leq (\geq) \; 0,
\]
for every nonnegative $\eta \in C_0^\infty(\QQ_R)$.

We say that $U_\vep$ is a weak solution in $\QQ_R$ if it is both a weak subsolution and supersolution, that is

\noindent $\bullet$ $U_\vep \in L^2(-R^2,R^2; H^{1,a}(\mathbb{B}_R))$ with $\partial_t U_\varepsilon \in L^2(-R^2,R^2:L^{2,a}(\mathbb{B}_R))$.

\noindent $\bullet$ $U_\vep$ satisfies
\begin{equation}\label{eq:WeakEquation}
\int_{\mathbb{Q}_R} |y|^a ( \varepsilon \partial_t U_\vep \partial_t \eta + \partial_t U_\vep \eta + \nabla U_\vep \cdot\nabla\eta - \tilde{F}_\vep \eta) \dX\rd t - \int_{Q_R} \tilde{f}_\vep \eta|_{y=0} \dx \rd t = 0,
\end{equation}
for every $\eta \in C_0^\infty(\mathbb{Q}_R)$.
\end{defn}
\begin{rem} In the whole section, even if not mentioned, we will always work with weak solutions (or subsolutions) in the sense of Definition \ref{def:WeakSolFf}. Further, to simplify the reading, we drop the notations $\tilde{F}_\vep$ and $\tilde{f}_\vep$, writing $F_\vep$ and $f_\vep$ instead. We stress that this does not change our estimates, since the extra factor $2$ can be easily reabsorbed through a dilation of the variables $(X,t)$.  
\end{rem}
\begin{rem}\label{rem:ScalingProp} (Scaling) In the proof of Proposition \ref{thm:CalphaBound} we will use that if $R > 0$ and $U_\vep$ is a weak solution in $\QQ_R$, then the function
\[
V_{\vep,R}(X,t) := U_\vep (R X,R^2 t), \quad (X,t) \in \QQ_1,
\]
satisfies
\[
\int_{\mathbb{Q}_1} |y|^a ( \tfrac{\varepsilon}{R^2} \partial_t V_{\vep,R} \partial_t \eta + \partial_t V_{\vep,R} \eta + \nabla V_{\vep,R} \cdot\nabla\eta - F_{\vep,R} \eta) \dX\rd t - \int_{Q_1} f_{\vep,R} \eta|_{y=0} \dx \rd t = 0,
\]
for every $\eta \in C_0^\infty(\QQ_1)$, that is $V_{\vep,R}$ is a weak solution in $\QQ_1$ (replacing $\vep$ with $\tfrac{\vep}{R^2}$), with
\[
\begin{aligned}
F_{\vep,R}(X,t) &:= R^2 F_\vep(RX,R^2t),  \\
f_{\vep,R}(x,t) &:= R^{1-a} f_\vep(Rx,R^2t).
\end{aligned}
\]
%
%which satisfy
%
%\[
%\begin{aligned}
%\|F_R\|_{L^{p,a}(\QQ_1)} &= R^{2 - \frac{N+3+a}{p}} \|F\|_{L^{p,a}(\QQ_R)},  \\
%
%\|f_R\|_{L_\infty^q(Q_1)} &= R^{1 - a - \frac{N}{q}} \|f\|_{L_\infty^q(Q_R)}.
%\end{aligned}
%\]
%
\end{rem}
We are ready to prove the energy estimate for families of nonnegative weak solutions (the same proof applies to family of weak solutions, see Remark \ref{rem:EnEst1}).
\begin{lem}\label{lem:EnergyEstimate} (Energy estimate)
Let $N \geq 1$, $a \in (-1,1)$ and $(p,q)$ satisfying \eqref{eq:AssumptionsExponents}. Then there exists a constant $C > 0$ depending only on $N$, $a$ and $q$ such that for every $r \in (\tfrac12,1]$, $\varrho \in [\tfrac12,r)$ and every family $\{U_\varepsilon\}_{\varepsilon \in (0,1)}$ of nonnegative subsolutions in $\QQ_1$ such that
\[
\|F_\vep\|_{L^{p,a}(\QQ_1)} + \|f_\vep\|_{L_\infty^q(Q_1)} \leq 1
\]
for every $\vep \in (0,1)$, we have
\begin{equation}\label{eq:EnergyEstimate}
\begin{aligned}
\esssup_{t \in (-\varrho^2,\varrho^2)} &\int_{\mathbb{B}_\varrho} |y|^a U_\vep^2(X,t) \dX + \int_{\mathbb{Q}_\varrho} |y|^a|\nabla U_\varepsilon|^2 \dX\rd t + \vep \int_{\mathbb{Q}_\varrho} |y|^a|\partial_t U_\varepsilon|^2 \dX\rd t \\
& \leq C \Big[ (r-\varrho)^{-2} \int_{\mathbb{Q}_r} |y|^a U_\varepsilon^2 dX\rd t + \|U_\vep\|_{L^{p',a}(\QQ_r)} \Big].
\end{aligned}
\end{equation}
\end{lem}
\begin{proof} Let $U = U_\varepsilon$, $F = F_\vep$, $f = f_\vep$ and $\tfrac{1}{2} \leq \varrho < r \leq 1$. Since $U$ is a weak subsolution, we may assume $F,f \geq 0$ (up to replacing them with $F_+$ and $f_+$, respectively) and so, by Lemma \ref{lem:IneqTruncations} (part (ii)), we may also assume $U \geq 1$ in $\QQ_1$ (up to consider the subsolution $\max\{U,1\}$ instead of $U$).

Let $\psi$ be a Lipschitz function vanishing on $\partial \mathbb{Q}_1$, which will be chosen later. Testing the differential inequality of $U$ with $\eta = U \psi^2$, we easily obtain the differential inequality
\begin{equation}\label{eq:L2LinfFirstDiffInequality}
\begin{aligned}
\tfrac{1}{2} &\int_{\mathbb{Q}_1} |y|^a \partial_t(U^2)\psi (\psi + 2\varepsilon\partial_t\psi)\dX\rd t +  \int_{\mathbb{Q}_1} |y|^a |\nabla U|^2 \psi^2 \dX\rd t + \varepsilon \int_{\mathbb{Q}_1} |y|^a |\partial_t U|^2 \psi^2 \dX\rd t \\
&\leq - 2\int_{\mathbb{Q}_1} |y|^a U \psi \nabla U \cdot \nabla \psi \dX\rd t + \int_{\mathbb{Q}_1} |y|^a F U \psi^2\dX\rd t + \int_{Q_1} f u \psi^2|_{y=0} \dX\rd t,
\end{aligned}
\end{equation}
where we recall that $u = U|_{y=0}$ in the sense of traces and $\psi|_{y=0}(x,t) = \psi(x,0,t)$.

The energy inequality \eqref{eq:EnergyEstimate} will be obtained combing two different bounds that we prove in two separate steps.

\smallskip

\emph{Step 1.} We prove that
\begin{equation}\label{eq:L2LinfFundEst1}
\esssup_{t \in (-\varrho^2,\varrho^2)} \int_{\mathbb{B}_\varrho} |y|^a U^2(X,t) \dX \leq \frac{\tilde{c}}{(r-\varrho)^2} \int_{\mathbb{Q}_r} |y|^a U^2 \rd X\rd t + \|U\|_{L^{p',a}(\QQ_r)},
\end{equation}
for some $\tilde{c} > 0$ depending only on $N$, $a$ and $q$.

Fix $t_\ast \in [-\varrho^2,\varrho^2]$ such that
\begin{equation}\label{eq:LinfL2BoundsupL2V2}
\int_{\mathbb{B}_\varrho} |y|^a U^2(X,t_\ast) \dX \geq \tfrac{1}{2} \, \esssup_{t \in (-\varrho^2,\varrho^2)} \int_{\mathbb{B}_\varrho} |y|^a U^2(X,t) \dX.
\end{equation}
Taking into account the relation
\[
|\nabla U|^2 \psi^2 + 2 U \psi \nabla U \cdot \nabla \psi = \sum_{j=1}^{N+1} \Big( \psi \partial_iU  + U \partial_i\psi \Big)^2 - |\nabla \psi|^2 U^2,
\]
we rewrite \eqref{eq:L2LinfFirstDiffInequality} neglecting the nonnegative term in the l.h.s. involving $\partial_t U$ to deduce
\begin{equation}\label{eq:L2LinfDiffIneqForEssSup}
\begin{aligned}
&\int_{\mathbb{Q}_1} |y|^a \partial_t(U^2)\psi (\psi + 2\varepsilon \partial_t \psi) \dX\rd t + \int_{\mathbb{Q}_1} |y|^a |\nabla U|^2 \psi^2 \dX\rd t \\
& \quad \leq 2 \int_{\mathbb{Q}_1} |y|^a |\nabla \psi|^2 U^2 \dX\rd t + \int_{\mathbb{Q}_1} |y|^a F U \psi^2\dX\rd t + \int_{Q_1} f u \psi^2|_{y=0} \dx\rd t.
\end{aligned}
\end{equation}
Let $\varphi \in C_0^\infty(\mathbb{B}_r)$, $0 \leq \varphi \leq 1$ with
\begin{equation}\label{eq:L2LinfPropVp}
\varphi = 1 \quad \text{ in } \mathbb{B}_\rho, \qquad |\nabla \varphi| \leq \frac{c_0}{r -\varrho},
\end{equation}
for some $c_0 > 0$ depending only on $N$, and $\phi = \phi(t)$ be defined by
\[
\phi(t) :=
\begin{cases}
0                                                          \quad &t \in [-1,-r^2]\cup[r^2,1] \\
\tfrac{r^2 + t}{r^2 -\varrho^2} \quad &t \in (-r^2,-\varrho^2) \\
1 \quad &t \in [-\varrho^2,t_\ast] \\
\alpha_\ast + (1-\alpha_\ast) e^{\frac{t_\ast-t}{2\varepsilon}} \quad &t \in (t_\ast,r^2),
\end{cases}
\qquad\qquad  \alpha_\ast := -\frac{1}{e^{\frac{r^2-t_\ast}{2\varepsilon}} - 1} < 0.
\]
An immediate computation shows that $\phi(-r^2) = \phi(r^2) = 0$, $\phi(t_\ast) = 1$, $0 \leq \phi \leq 1$ and
\begin{equation}\label{eq:L2LinfPropPh}
\phi + 2\varepsilon\phi' = \alpha_\ast \quad \text{ in } (t_\ast,r^2).
\end{equation}
Furthermore, since $\tfrac{1}{e^z-1} \leq \tfrac{2}{z}$ for every $z > 0$, and $t_\ast \leq \varrho^2$, we have
\begin{equation}\label{eq:L2LinfPropAaPhiPr}
|\alpha_\ast| \leq \frac{4\varepsilon}{r^2-t_\ast} \leq \frac{8\varepsilon}{r-\varrho}, \qquad |\phi'| \leq \frac{1+|\alpha_\ast|}{2\varepsilon} \quad \text{ in } (t_\ast,r^2).
\end{equation}
Now, choose $\psi(X,t) = \varphi(X) \phi(t)$ and write
\[
\begin{aligned}
\int_{\mathbb{Q}_1} |y|^a \partial_t(U^2)\psi (\psi + 2\varepsilon \partial_t \psi) \dX\rd t &= \int_{-1}^{-\varrho^2} \!\! \int_{\mathbb{B}_1} |y|^a \partial_t(U^2)\psi (\psi + 2\varepsilon \partial_t \psi) \dX\rd t\\
&\quad +  \int_{-\varrho^2}^{t_\ast} \int_{\mathbb{B}_1} |y|^a \partial_t(U^2)\psi (\psi + 2\varepsilon \partial_t \psi) \dX\rd t\\
&\quad + \int_{t_\ast}^{1} \int_{\mathbb{B}_1} |y|^a \partial_t(U^2)\psi (\psi + 2\varepsilon \partial_t \psi) \dX\rd t := I_1 + I_2 + I_3.
\end{aligned}
\]
Using the definitions of $\phi$ and $\phi'$ in $[-1,t_\ast]$, that $\varphi \in C_0^\infty(\mathbb{B}_r)$ with $0 \leq \varphi \leq 1$ and integrating by parts, we find
\[
\begin{aligned}
I_1 &= \int_{-r^2}^{-\varrho^2} \!\!\int_{\mathbb{B}_r} |y|^a \partial_t(U^2) \varphi^2 \phi (\phi + 2\varepsilon \phi') \dX\rd t \\
&\geq -2\int_{-r^2}^{-\varrho^2} (\phi\phi' + \varepsilon (\phi')^2 ) \int_{\mathbb{B}_r} |y|^a  U^2 \varphi^2  \dX\rd t + \int_{\mathbb{B}_r} |y|^a  U^2(X,-\varrho^2) \varphi^2  \dX \\
&\geq - \tfrac{8}{(r-\varrho)^2} \int_{\mathbb{Q}_r} |y|^a U^2 \dX\rd t +  \int_{\mathbb{B}_r} |y|^a  U^2(X,-\varrho^2) \varphi^2  \dX,
\end{aligned}
\]
and
\[
I_2 = \int_{-\varrho^2}^{t_\ast} \int_{\mathbb{B}_r} |y|^a \partial_t(U^2)\varphi^2 \dX\rd t = \int_{\mathbb{B}_r} |y|^a U^2(X,t_\ast)\varphi^2(X) \dX - \int_{\mathbb{B}_r} |y|^a U^2(X,-\varrho^2)\varphi^2(X) \dX.
\]
Further, by \eqref{eq:L2LinfPropPh} and \eqref{eq:L2LinfPropAaPhiPr},
\[
\begin{aligned}
I_3&= \alpha_\ast \int_{t_\ast}^{r^2} \!\! \int_{\mathbb{B}_r} |y|^a \partial_t(U^2) \varphi^2 \phi  \dX\rd t = |\alpha_\ast| \int_{\mathbb{B}_r} |y|^a U^2(X,t_\ast) \varphi^2(X) \dX + |\alpha_\ast| \int_{t_\ast}^{r^2} \!\! \int_{\mathbb{B}_r} |y|^a U^2 \varphi^2 \phi'  \dX\rd t \\
&\geq -|\alpha_\ast| \int_{t_\ast}^{r^2} \!\! \int_{\mathbb{B}_r} |y|^a U^2 \varphi^2 |\phi'|  \dX\rd t \geq - \tfrac{36}{(r-\varrho)^2} \int_{\mathbb{Q}_r} |y|^a U^2 \dX\rd t,
\end{aligned}
\]
and thus by \eqref{eq:LinfL2BoundsupL2V2} and \eqref{eq:L2LinfPropVp}
\begin{equation}\label{eq:L2LinfEst1}
I_1 + I_2 + I_3 \geq \tfrac{1}{2} \, \esssup_{t \in (-\varrho^2,\varrho^2)} \int_{\mathbb{B}_\varrho} |y|^a U^2(X,t) \dX - \tfrac{44}{(r-\varrho)^2} \int_{\mathbb{Q}_r} |y|^a U^2 \dX\rd t.
\end{equation}
Now let us estimate the terms into the r.h.s. of \eqref{eq:L2LinfDiffIneqForEssSup}. First, we have
\begin{equation}\label{eq:L2LinfEst3}
\int_{\mathbb{Q}_1} |y|^a |\nabla \psi|^2 U^2 \dX\rd t \leq \int_{\mathbb{Q}_r} |y|^a |\nabla \varphi|^2 \phi^2 U^2 \dX\rd t \leq \tfrac{c_0}{(r-\varrho)^2} \int_{\mathbb{Q}_r} |y|^a U^2 \rd X\rd t,
\end{equation}
thanks to \eqref{eq:L2LinfPropVp} and the fact that $\phi = 0$ in $[-1,-r^2]\cup[r^2,1]$. Second, by H\"older's inequality
\begin{equation}\label{eq:EnEstBoundLqFterm}
\int_{\mathbb{Q}_1} |y|^a F U \psi^2\dX\rd t  \leq \int_{\mathbb{Q}_r} |y|^a F U \dX\rd t \leq \|F\|_{L^{p,a}(\QQ_r)} \|U\|_{L^{p',a}(\QQ_r)} \leq \|U\|_{L^{p',a}(\QQ_r)},
\end{equation}
since $\|F\|_{L^{p,a}(\QQ_1)} \leq 1$ (and $F \geq 0$) by assumption.

We are left to estimate the trace term, that we reabsorb using the second term in the l.h.s. of \eqref{eq:L2LinfDiffIneqForEssSup}. Using that $u \geq 1$ and $f \geq 0$, applying H\"older's inequality and recalling that $\|f\|_{L_\infty^q(Q_1)} \leq 1$, we obtain
\[
\begin{aligned}
\int_{Q_1} f u \psi^2|_{y=0} \dx\rd t &\leq \int_{Q_1} f (u \psi|_{y=0})^2 \dx\rd t \leq \|f\|_{L_\infty^q(Q_1)} \|v^2\|_{L_1^{q'}(Q_1)} \leq \|v^2\|_{L_1^{q'}(Q_1)} = \|v\|_{L_2^{2q'}(Q_1)},
\end{aligned}
\]
where we have set $v := u \psi_{y=0}$, and $q'$ is the conjugate of $q$. Since $q > \tfrac{N}{1-a}$, we have $2 \leq 2q' \leq 2\tilde{\sigma}$, where 
$\tilde{\sigma} := \tfrac{N}{N-1+a}$ (cf. Theorem \ref{thm:TraceThm}) and so, by interpolation and Young's inequality, we obtain 
\[
\|v\|_{L_2^{2q'}(Q_1)} \leq \int_{-1}^1 \|v(\cdot,t)\|_{L^2(B_1)}^{2\vartheta} \|v(\cdot,t)\|_{L^{2\tilde{\sigma}}(B_1)}^{2(1-\vartheta)} \rd t \leq \delta \| v \|_{L^{2\tilde{\sigma}}_2(Q_1)}^2 + c_\delta \| v \|_{L^2_2(Q_1)},
\]
for every $\delta > 0$ and a suitable $c_\delta > 0$, satisfying $c_\delta \to +\infty$ as $\delta \to 0$ ($\vartheta \in (0,1)$ is given in the interpolation inequality and depends on $N$, $a$ and $q$). Now, by \eqref{eq:EmbeddingTrace}, the definition of $v$ and Cauchy-Schwartz's inequality, we have
\[
\begin{aligned}
\| v \|_{L^{2\tilde{\sigma}}_2(Q_1)}^2 &\leq c \int_{\mathbb{Q}_1} |y|^a |\nabla (U\psi)|^2 \rd X \rd t \leq 2c \Big( \int_{\mathbb{Q}_1} |y|^a |\nabla U|^2 \psi^2 \rd X \rd t + \int_{\mathbb{Q}_1} |y|^a U^2 |\nabla \psi|^2 \rd X \rd t \Big),
\end{aligned}
\]
for some $c > 0$ depending only in $N$ and $a$, while, by \eqref{eq:TraceIneqWithEps} and the Cauchy-Schwartz's inequality again, 
\[
\begin{aligned}
\| v \|_{L^2_2(Q_1)} &\leq c \Big( A^{\frac{1+a}{2}} \int_{\mathbb{Q}_1} |y|^a U^2 \psi^2 \rd X \rd t +  A^{-\frac{1-a}{2}} \int_{\mathbb{Q}_1} |y|^a |\nabla (U\psi)|^2 \rd X \rd t  \Big) \\
&\leq 2c \Big( A^{\frac{1+a}{2}} \int_{\mathbb{Q}_1} |y|^a U^2 \big( \psi^2 + |\nabla \psi|^2\big)  \rd X \rd t +  A^{-\frac{1-a}{2}} \int_{\mathbb{Q}_1} |y|^a |\nabla U|^2 \psi^2 \rd X \rd t  \Big)
\end{aligned}
\]
for every $A > 1$. Now, we fix $\delta \in (0,1)$, such that $2c\delta \leq \tfrac12$ and $A > 1$ such that $2cc_\delta A^{-\frac{1-a}{2}} \leq \tfrac12$ (notice that both $\delta$ and $A$ depend only on $N$, $a$ and $q$). Combing the last four inequalities with \eqref{eq:L2LinfDiffIneqForEssSup}, we obtain
\[
\begin{aligned}
&\int_{\mathbb{Q}_1} |y|^a \partial_t(U^2)\psi (\psi + 2\varepsilon \partial_t \psi) \dX\rd t \leq c \int_{\mathbb{Q}_1} |y|^a  U^2 \big( \psi^2 + |\nabla \psi|^2\big) \dX\rd t + \int_{\mathbb{Q}_1} |y|^a F U \psi^2\dX\rd t,
\end{aligned}
\]
for some new $c > 0$ depending only on $N$, $a$ and $q$, and thus \eqref{eq:L2LinfFundEst1} follows in light of \eqref{eq:L2LinfEst1}, \eqref{eq:L2LinfEst3} and \eqref{eq:EnEstBoundLqFterm}.

\smallskip

\emph{Step 2.} In this second step we show
\begin{equation}\label{eq:L2LinfFundEst2}
\int_{\mathbb{Q}_{\varrho}} |y|^a|\nabla U|^2 \dX\rd t + \vep \int_{\mathbb{Q}_{\varrho}} |y|^a|\partial_t U|^2 \dX\rd t \leq \frac{\overline{c}}{(r-\varrho)^2} \int_{\mathbb{Q}_r} |y|^a U^2 \dX\rd t + \|U\|_{L^{p',a}(\QQ_r)},
\end{equation}
for some $\overline{c} > 0$ depending only on $N$, $a$ and $q$.

Test \eqref{eq:L2LinfFirstDiffInequality} with $\psi_0(X,t) = \varphi(X) \phi_0(t)$, where $\varphi$ is as in \eqref{eq:L2LinfPropVp}, while $\phi_0 \in C_0^\infty([-r^2,r^2])$, $0 \leq \phi_0 \leq 1$ and
\begin{equation}\label{eq:Test2L2Linf}
\phi_0 = 1 \quad \text{ in } [-\varrho^2,\varrho^2], \qquad \|\phi_0'\|_{L^\infty(0,1)} \leq \frac{c_1}{r-\varrho},
\end{equation}
for some numerical constant $c_1 \geq 1$. Integrating by parts w.r.t. time the first term in the l.h.s. of \eqref{eq:L2LinfFirstDiffInequality}, it follows
\[
\int_{\mathbb{Q}_1} |y|^a \partial_t(U^2)\psi_0^2\dX\rd t = -2 \int_{\mathbb{Q}_r} |y|^a U^2 \varphi^2(X) \phi_0(t)\phi_0'(t) \dX\rd t,
\]
since $\varphi \in C_0^\infty(\mathbb{B}_r)$ and $\phi_0 \in C_0^\infty([-r^2,r^2])$ and thus, we may write \eqref{eq:L2LinfFirstDiffInequality} as
\[
\begin{aligned}
& \int_{\mathbb{Q}_1} |y|^a |\nabla U|^2 \psi_0^2 \dX\rd t + \varepsilon \int_{\mathbb{Q}_1} |y|^a |\partial_t U|^2 \psi_0^2 \dX\rd t  \\
\leq & -2\int_{\mathbb{Q}_1} |y|^a U \psi_0 \nabla U \cdot \nabla \psi_0 \dX\rd t - 2\vep \int_{\mathbb{Q}_1} |y|^a U \psi_0 \partial_t U \partial_t \psi_0 \dX\rd t \\
& + \int_{\mathbb{Q}_r} |y|^a U^2 \varphi^2(X) \phi_0(t)\phi_0'(t) \rd X\rd t + \int_{\mathbb{Q}_1} |y|^a F U \psi_0^2\dX\rd t + \int_{Q_1} f u \psi_0^2|_{y=0} \dx\rd t.
\end{aligned}
\]
Now, using Cauchy-Schwartz's inequality ($2AB \leq 2\ep A^2 + \tfrac{1}{2\ep} B^2$ with $\ep = \tfrac12$), \eqref{eq:L2LinfPropVp}-\eqref{eq:Test2L2Linf} and H\"older's inequality, we obtain
\begin{equation}\label{eq:L2LinfGradV}
\begin{aligned}
&\int_{\mathbb{Q}_1} |y|^a |\nabla U|^2 \psi_0^2 \dX\rd t + \varepsilon \int_{\mathbb{Q}_1} |y|^a |\partial_t U|^2 \psi_0^2 \dX\rd t \\
&\leq 4 \int_{\mathbb{Q}_r} |y|^a U^2 \left[\varphi^2 \phi_0 |\phi_0'| + |\nabla \varphi|^2 \phi_0^2 + \vep \varphi^2 (\phi_0')^2 \right] \rd X\rd t + \|U\|_{L^{p',a}(\QQ_r)} + \int_{Q_1} f u \psi_0^2|_{y=0} \dx\rd t\\
&\leq \frac{4(c_1 + c_1^2 + c_0^2)}{(r-\varrho)^2} \int_{\mathbb{Q}_r} |y|^a U^2 \dX\rd t + \|U\|_{L^{p',a}(\QQ_r)} + \int_{Q_1} f u \psi_0^2|_{y=0} \dx\rd t,
\end{aligned}
\end{equation}
and so, reabsorbing the trace term as we have done in \emph{Step 1}, and using \eqref{eq:L2LinfPropVp}-\eqref{eq:Test2L2Linf} again, \eqref{eq:L2LinfFundEst2} follows.

\smallskip

\emph{Step 3.} The energy estimate \eqref{eq:EnergyEstimate} follows by summing \eqref{eq:L2LinfFundEst1} and \eqref{eq:L2LinfFundEst2}, and taking $C = \tilde{c} + \overline{c}$.
\end{proof}
\begin{rem}\label{rem:EnEst1} The same proof works if we assume that $\{U_\varepsilon\}_{\varepsilon \in (0,1)}$ is a family of solutions, without sign restrictions.
\end{rem}
\begin{rem}\label{rem:ParEE} The energy estimate \eqref{eq:EnergyEstimate} is the main step for proving the $L^2 \to L^\infty$ bound \eqref{eq:L2LinfEstimate}. In light of what comes next, we believe it is important to compare the proof w.r.t. the classical parabolic framework (formally, the limit case $\vep=0$). Assume for simplicity $F=0$ and $f=0$ and consider a weak solution to 
\begin{equation}\label{eq:ParabolicEqRem}
|y|^a\partial_t U - \nabla\cdot(|y|^a\nabla U) = 0 \quad \text{ in } \QQ_1.
\end{equation}
Testing the weak formulation with $\eta := U \psi^2$, we obtain 
\[
\tfrac{1}{2} \int_{\mathbb{Q}_1} |y|^a \partial_t(U^2)\psi^2 \dX\rd t +  \int_{\mathbb{Q}_1} |y|^a |\nabla U|^2 \psi^2 \dX\rd t \leq - 2\int_{\mathbb{Q}_1} |y|^a U \psi \nabla U \cdot \nabla \psi \dX\rd t,
\]
which is \eqref{eq:L2LinfFirstDiffInequality} ``with $\vep = 0$''. For every fixed $-1<s<\tau<1$, one may choose $\psi^2(X,t) = \varphi^2(X) \chi_{[s,\tau]}(t)$ and so, integrating by parts w.r.t. time and using Young's inequality, it is not difficult to find
\begin{equation}\label{eq:EnEs1HeatEq}
\int_{\mathbb{B}_1} |y|^a U^2(X,\tau)\varphi^2(X) \dX - \int_{\mathbb{B}_1} |y|^a U^2(X,s)\varphi^2(X) \dX \leq C \int_s^\tau\int_{\mathbb{B}_1} |y|^a U^2 |\nabla \varphi|^2 \dX\rd t,
\end{equation}
for some $C > 0$ (depending only on $N$ and $a$). The bound \eqref{eq:L2LinfFundEst1} immediately follows from \eqref{eq:EnEs1HeatEq}, choosing $\varphi$ as in \eqref{eq:L2LinfPropVp} and neglecting the nonpositive term in the l.h.s. The $L^{2,a}$ bound for $\nabla U$ is obtained  similar to \eqref{eq:L2LinfFundEst2}.

In our setting testing with $\psi^2 = \varphi^2(X)\chi_{[s,\tau]}(t)$ is not admissible since the weak derivative of the function $t \to \chi_{[s,\tau]}(t)$ is not $L^2(-1,1)$. However, we notice that our proof is somehow more elementary: the time factor in the test function $\psi = \varphi(X)\phi(t)$ we use to prove the bound for $\int_{\mathbb{B}_1} |y|^a U^2(X,\tau) \dX$ is obtained by solving an easy first order ODE. In this way, we bypass the approximation procedure of $\chi_{[s,\tau]}$ needed in the purely parabolic framework.
\end{rem}
\begin{lem}\label{lem:NoSpikesEstimate} (No-spikes estimate)
Let $N \geq 1$, $a \in (-1,1)$ and $(p,q)$ satisfying \eqref{eq:AssumptionsExponents}. Then there exists a constant $\delta > 0$ depending only on $N$, $a$ and $q$ such that for every family $\{U_\varepsilon\}_{\varepsilon \in (0,1)}$ of subsolutions in $\QQ_1$ such that
\begin{equation}\label{eq:NoSpikesL2Bound}
\|F_\vep\|_{L^{p,a}(\QQ_1)} + \|f_\vep\|_{L_\infty^q(Q_1)} \leq 1 \qquad \text{ and } \qquad \int_{\QQ_1} |y|^a (U_\vep)_+^2 \dX \rd t \leq \dd, 
\end{equation}
for every $\vep \in (0,1)$, then
\begin{equation}\label{eq:NoSpikesL2Linf}
U_\vep \leq 1 \quad \text{ in } \QQ_{1/2}, 
\end{equation}
for every $\vep \in (0,1)$.
\end{lem}
\begin{proof} Let us set $U = U_\vep$, $F_\vep = F$, $f_\vep = f$ and assume either $N \geq 2$ and $a \in (-1,1)$, or $N = 1$ and $a \in [0,1)$. For every integer $j \geq 0$, define 
\[
\begin{aligned}
\tilde{\BB}_j := \BB_{r_j}, \qquad \tilde{\QQ}_j := \QQ_{r_j}, \qquad  r_j := \tfrac12 + 2^{-j-1},
\end{aligned}
\]
the nonnegative subsolutions (with $F_+$ and $f_+$, see Lemma \ref{lem:IneqTruncations})
\[
\begin{aligned}
V_j := ( U - C_j )_+, \qquad C_j := 1 - 2^{-j},
\end{aligned}
\]
and the quantity
\[
E_j := \int_{\tilde{\QQ}_j} |y|^a V_j^2 \dX\rd t.
\]
Applying the energy inequality \eqref{eq:EnergyEstimate} to $V_{j+1}$ with $\varrho = r_{j+1}$ and $r = r_j$, we have
\[
\begin{aligned}
\tfrac{1}{r_{j+1}^2} \int_{\tilde{\mathbb{Q}}_{j+1}} |y|^a V_{j+1}^2 \dX\rd t &+ \esssup_{t \in (-r_{j+1}^2,r_{j+1}^2)} \int_{\tilde{\mathbb{B}}_{j+1}} |y|^a V_{j+1}^2(X,t) \dX + \int_{\tilde{\mathbb{Q}}_{j+1}} |y|^a|\nabla V_{j+1}|^2 \dX\rd t \\
& \leq C \Big[ (r_j-r_{j+1})^{-2} \int_{\tilde{\mathbb{Q}}_j} |y|^a V_{j+1}^2 \dX\rd t + \|V_{j+1}\|_{L^{p',a}(\tilde{\QQ}_j)} \Big],
\end{aligned}
\]
for some $C > 0$ depending only on $N$, $a$ and $q$ and thus, by Sobolev inequality (cf. Theorem \ref{thm:SobEmbedLocalPar}, formula \eqref{eq:ParSobIneq1}) and the definition of $V_j$, we deduce
\begin{equation}\label{eq:NoSpikesSobolev}
\begin{aligned}
\Big( \int_{\tilde{\mathbb{Q}}_{j+1}} |y|^a V_{j+1}^{2\gamma} \rd X\rd t \Big)^{\frac{1}{\gamma}} &\leq C 2^{2j} \int_{\tilde{\mathbb{Q}}_j} |y|^a V_{j+1}^2 \rd X\rd t + C\|V_{j+1}\|_{L^{p',a}(\tilde{\QQ}_j)} \\
&\leq C 2^{2j} \int_{\tilde{\mathbb{Q}}_j} |y|^a V_j^2 \dX\rd t + C\|V_{j+1}\|_{L^{p',a}(\tilde{\QQ}_j)}, 
\end{aligned}
\end{equation}
for some new $C > 0$ and $\gamma > 1$ (depending only on $N$, $a$ and $q$). On the other hand, by H\"older's inequality
\begin{equation}\label{eq:NoSpikesHolder}
E_{j+1} = \int_{\tilde{\QQ}_{j+1}} |y|^a V_{j+1}^2 \dX\rd t \leq \Big( \int_{\tilde{\mathbb{Q}}_{j+1}} |y|^a V_{j+1}^{2\gamma} \rd X\rd t \Big)^{\frac{1}{\gamma}} \cdot |\{ V_{j+1} > 0 \} \cap \tilde{\mathbb{Q}}_{j+1} |_a^{\frac{1}{\gamma'}},
\end{equation}
where $\gamma'$ is the conjugate exponent of $\gamma$, and $|A|_a := \int_A |y|^a \dX\rd t$ for every measurable set $A \subset \RR^{N+2}$. Further, using the definition of $V_j$, it follows
\begin{equation}\label{eq:NoSpikesBoundLevelSets}
\begin{aligned}
|\{ V_{j+1} > 0 \} \cap \tilde{\mathbb{Q}}_{j+1} |_a &= |\{ V_j > 2^{-j-1} \} \cap \tilde{\mathbb{Q}}_{j+1} |_a = |\{ V_j^2 > 2^{-2j-2} \} \cap \tilde{\mathbb{Q}}_{j+1} |_a \\
& \leq 2^{2j + 2} \int_{\tilde{\QQ}_{j+1}} |y|^a V_j^2 \dX\rd t \leq 2^{2j + 2} \int_{\tilde{\QQ}_j} |y|^a V_j^2 \dX\rd t = 2^{2j + 2} E_j,
\end{aligned}
\end{equation}
and, by H\"older's inequality again,
\begin{equation}\label{eq:NoSpikesLpBound}
\|V_{j+1}\|_{L^{p',a}(\tilde{\QQ}_j)} \leq \Big( \int_{\tilde{\QQ}_j} |y|^a V_{j+1}^2 \dX\rd t \Big)^{\frac{1}{2}} |\{ V_{j+1} > 0 \} \cap \tilde{\mathbb{Q}}_{j+1} |_a^{\frac{p-2}{2p}} \leq E_j^{\frac{1}{2}} |\{ V_{j+1} > 0 \} \cap \tilde{\mathbb{Q}}_{j+1} |_a^{\frac{p-2}{2p}}.
\end{equation}
So, combining \eqref{eq:NoSpikesSobolev}, \eqref{eq:NoSpikesHolder}, \eqref{eq:NoSpikesBoundLevelSets} and \eqref{eq:NoSpikesLpBound}, and using \eqref{eq:NoSpikesL2Bound}, we obtain
\[
\begin{cases}
E_{j+1} \leq C^{1+j} \big( E_j^{1 + 1/\gamma'} +  E_j^{1 + \overline{\gamma}} \big) \\
E_0 \leq \dd,
\end{cases}
\]
for every $j \geq 1$ and for some new $C > 0$ depending only on $N$, $a$ and $p$, where $\overline{\gamma} := \tfrac{1}{\gamma'} - \tfrac{1}{p} > 0$ (in view of \eqref{eq:AssumptionsExponents} and the definition of $\gamma$, cf. Theorem \ref{thm:SobEmbedLocalPar}).

To complete the argument, it is sufficient to notice that since $\overline{\gamma} < \tfrac{1}{\gamma'}$, we have $E_{j+1} \leq C^{1+j} E_j^{1 + \overline{\gamma}}$ and
\[
E_j \leq C^{(1 + \overline{\gamma})^j  \big(\sum_{i=1}^j \frac{i}{(1+ \overline{\gamma})^i} \big)} \delta^{(1+\overline{\gamma})^j} \leq C^{(1 + \overline{\gamma})^j  \big(\sum_{i=1}^\infty \frac{i}{(1+ \overline{\gamma})^i} \big)} \delta^{(1+\overline{\gamma})^j} \leq (C\delta)^{(1+\overline{\gamma})^j},
\]
for some new $C > 0$ (depending only on $N$ and $a$), whenever $E_j \leq 1$. Thus, choosing $\dd$ such that $C\dd < 1$, we deduce $E_j \leq 1$ for every $j$ and $E_j \to 0$ as $j \to +\infty$. Finally, since $C_j \to 1$ and $r_j \to \tfrac12$ as $j \to + \infty$, we obtain by definition of $V_j$
\[
E_j \to \int_{\QQ_{1/2}} (U-1)_+^2 \dX \rd t = 0,
\]
which implies \eqref{eq:NoSpikesL2Linf} and our statement follows.

The very same argument works when $N=1$ and $a \in (-1,0)$, by taking $\gamma = 2$ and using inequality \eqref{eq:ParSobIneq2} instead of \eqref{eq:ParSobIneq1}. Notice that in this range of parameters we have $p > \bar{p} = 2$: this implies $p' \in (1,2)$ and thus the chain of inequalities in \eqref{eq:NoSpikesLpBound} does not change.
\end{proof}
\begin{proof}[Proof of Proposition \ref{prop:L2LinfEstimate}] Set $U = U_\vep$ and define
\[
V := \lambda_+ U_+, \qquad \lambda_+ := \frac{\sqrt{\dd}}{\|U_+\|_{L^{2,a}(\QQ_1)} + \|F_+\|_{L^{p,a}(\QQ_1)} + \|f_+\|_{L_\infty^q(Q_1)}},
\]
where $\dd > 0$ is as in Lemma \ref{lem:NoSpikesEstimate}. Since $V$ is a nonnegative subsolution (with $\lambda_+F_+$ and $\lambda_+f_+$) and satisfies \eqref{eq:NoSpikesL2Bound}, then \eqref{eq:NoSpikesL2Linf} gives
\[
\|U_+\|_{L^\infty(\QQ_{1/2})} \leq \tfrac{1}{\sqrt{\dd}} \big(  \|U_+\|_{L^{2,a}(\QQ_1)} + \|F_+\|_{L^{p,a}(\QQ_1)} + \|f_+\|_{L_\infty^q(Q_1)} \big).
\]
Repeating the same argument with 
\[
V := \lambda_- U_-, \qquad \lambda_- := \frac{\sqrt{\dd}}{\|U_-\|_{L^{2,a}(\QQ_1)} + \|F_-\|_{L^{p,a}(\QQ_1)} + \|f_-\|_{L_\infty^q(Q_1)}},
\]
we obtain
\[
\|U_-\|_{L^\infty(\QQ_{1/2})} \leq \tfrac{1}{\sqrt{\dd}} \big(  \|U_-\|_{L^{2,a}(\QQ_1)} + \|F_-\|_{L^{p,a}(\QQ_1)} + \|f_-\|_{L_\infty^q(Q_1)} \big),
\]
and our thesis follows choosing $C = \tfrac{2}{\sqrt{\dd}}$, recalling that $\dd$ depends only on $N$, $a$ and $q$.
\end{proof}
\begin{rem}\label{rem:Translation1} Thanks to \cite[Remark 2.1, Lemma 2.1]{ChiarenzaSerapioni1985:art}, our proofs can be easily adapted to treat the case of weak solutions/sub-solutions in $\QQ_1(X_0,t_0)$, with constants independent of $(X_0,t_0) \in\RR^{N+2}$.
\end{rem}
%
%
%
%
%

%
%
%%%%%%%%%%%%%%%%%%%%%%%%%%%%%%%%%%%%%%%%%%%%%%%%%%%%%%%%%%%%%%%%%%%%%%%%%%%%%%%%%%%%%%%%%%%%%%%%%%%%%%%%%%%%%%
%
%%%%%%%%%%%%%%%%%%%%%%%%%%%%%%%%%%%%%%%%%%%%%%%%%%%%%%%%%%%%%%%%%%%%%%%%%%%%%%%%%%%%%%%%%%%%%%%%%%%%%%%%%%%%%%
%
%

%
\section{Uniform H\"older estimates}\label{Sec:OscDecay}
This section is devoted to the proof of a uniform H\"older estimate for families of solutions to \eqref{eq:PerProbEpsFf} (in the sense of Definition \ref{def:WeakSolFf}). As mentioned in the introduction, we will consider sequences of weak solutions satisfying the following compactness assumption 
\begin{equation}\label{eq:ConvergenceOscDecay}
U_{\vep_j} \to U \quad \text{ in } C([-r^2,r^2]:L^{2,a}(\BB_r)) \quad \text{as } j \to +\infty,
\end{equation}
for some $r > 0$ and some $U \in C([-r^2,r^2]:L^{2,a}(\BB_r))$. As pointed out in Remark \ref{rem:CompactnessPar}, in the classical parabolic framework this property can be deduced working directly with the equation and exploiting the parabolic energy estimate. Even though we do not have any counter-example, this approach seems not work in our setting. Anyway, in light of Proposition \ref{prop:weaklimit}, \eqref{eq:ConvergenceOscDecay} is always guaranteed when working with sequences $\{U_{\vep_j}\}_{j\in\NN}$ of minimizers of the functional \eqref{eq:Functional}.

The following proposition is the main result of this section.
\begin{prop}\label{thm:CalphaBound}
Let $N \geq 1$, $a \in (-1,1)$ and $(p,q)$ satisfying \eqref{eq:AssumptionsExponents}. There exist $\alpha \in (0,1)$ and $C > 0$ depending only on $N$, $a$, $p$ and $q$ such that for every $U$, every sequence $\vep_j \to 0$ and every sequence $\{U_{\vep_j}\}_{j \in \NN}$ of weak solutions in $\QQ_8$ satisfying \eqref{eq:ConvergenceOscDecay} with $r=4$ and
\begin{equation}\label{eq:FinalBoundFjfj}
\|F_{\vep_j}\|_{L^{p,a}(\QQ_4)} + \|f_{\vep_j}\|_{L_\infty^q(Q_4)} \leq C_0
\end{equation}
for every $j \in \NN$ and some $C_0 > 0$, then there exist a subsequence $j_k \to +\infty$ such that 
\begin{equation}\label{eq:HolderBoundsEstFf}
\| U_{\vep_{j_k}} \|_{C^{\alpha,\alpha/2}(\QQ_1)} \leq C \big( \| U_{\vep_{j_k}} \|_{L^\infty(\QQ_4)} + \|F_{\vep_{j_k}}\|_{L^{p,a}(\QQ_4)} + \|f_{\vep_{j_k}}\|_{L_\infty^q(Q_4)} \big),
\end{equation}
for every $k \in \NN$.
\end{prop}
The estimate \eqref{eq:HolderBoundsEstFf} will follow as a consequence of an oscillation decay type result which, in turn, is obtained by combining the $L^{2,a} \to L^\infty$ bound for nonnegative subsolutions (see Proposition \ref{prop:L2LinfEstimate}) and a parabolic version of the ``De Giorgi isoperimetric lemma'' (see Lemma \ref{lem:IsoperimetricIneq}). Lemma \ref{lem:IsoperimetricIneq} is the key step: our approach is based on the ideas of \cite[Section 3]{Vasseur2016:art} and relies on the validity of its ``elliptic'' counterpart, that we state in Lemma \ref{lem:IsoperimetricIneqElliptic} (the proof is a modification of \cite[Lemma 10]{Vasseur2016:art}).
\begin{lem}\label{lem:IsoperimetricIneqElliptic} (``De Giorgi isoperimetric lemma'')
Let $N \geq 1$, $a \in (-1,1)$. Let $U \in H^{1,a}(\BB_1)$ satisfying
\[
\int_{\BB_1} |y|^a |\nabla U|^2 \dX \leq C_0,
\]
for some $C_0 > 0$ and
\[
A := \{U \geq 1/2 \} \cap \BB_1, \qquad C := \{U \leq 0 \} \cap \BB_1, \qquad
D := \{0 < U < 1/2 \} \cap \BB_1.
\]
Then for every $p \in (1,2)$, there exists a constant $c > 0$ depending only on $N$, $a$, $C_0$ and $p$ such that
\[
|A|_a \cdot |C|_a \leq c |D|_a^{\frac{2-p}{2p}}.
\]
\end{lem}
\begin{proof} We consider a new function $V$ defined as $V = U$ in $D$, $V = \tfrac12$ in $A$ and $V = 0$ in $C$. Clearly, it satisfies $\nabla V = 0$ in $\BB_1\setminus D$ and $\int_{\BB_1} |y|^a |\nabla V|^2 \dX \leq C_0$.

Now, setting $\overline{V}_a := |\BB_1|_a^{-1} \int_{\BB_1} |y|^a V(X) \dX$ and writing $X = (x,y)$ and $Z = (z,\zeta)$, we have
\[
\begin{aligned}
|A|_a \cdot |C|_a &= 2 \int_C |\zeta|^a \dZ \cdot \int_A |y|^a V(X) \dX \leq 2 \int_C \int_A |y|^a |\zeta|^a |V(X) - V(Z)| \dX \dZ \\
& \leq 2 \iint_{\BB_1^2} |y|^a |\zeta|^a |V(X) - \overline{V}_a| \dX \dZ + 2 \iint_{\BB_1^2} |y|^a |\zeta|^a |V(Z) - \overline{V}_a| \dX \dZ \\
& \leq 4 |\BB_1|_a \int_{\BB_1} |y|^a |V(X) - \overline{V}_a| \dX \leq  4 |\BB_1|_a^{1 + 1/p'} \Big( \int_{\BB_1} |y|^a |V(X) - \overline{V}_a|^p \dX \Big)^{1/p},
\end{aligned}
\]
where, in the last inequality, we have used H\"older's inequality. By \cite[Theorem 1.5]{FabesKS1982:art} and H\"older's inequality again, we have
\[
\begin{aligned}
\int_{\BB_1} |y|^a |V(X) - \overline{V}_a|^p \dX &\leq c \int_{\BB_1} |y|^a |\nabla V|^p \dX = c \int_D |y|^a |\nabla V|^p \dX \\
&\leq c \Big(\int_{\BB_1} |y|^a |\nabla V|^2 \dX \Big)^{p/2} |D|_a^{\frac{2-p}{2}} \leq c \, C_0^{p/2} |D|_a^{\frac{2-p}{2}},
\end{aligned}
\]
for some $c > 0$ depending on $p$. Combining the two inequalities, our statement follows.
\end{proof}
\begin{lem}\label{lem:IsoperimetricIneq}
Let $N \geq 1$, $a \in (-1,1)$, $(p,q)$ satisfying \eqref{eq:AssumptionsExponents}, and $\tilde{\QQ} := \BB_1\times(-2,-1)$.

Assume there exist $\dd > 0$, a sequence $\vep_j \to 0$ and a family $\{U_{\vep_j}\}_{j \in \NN}$ of weak solutions in $\QQ_2$ such that
\[
U_{\vep_j} \leq 1, \qquad |\{U_{\vep_j} \geq 1/2 \} \cap \QQ_1|_a \geq \dd, \qquad |\{U_{\vep_j} \leq 0 \} \cap \tilde{\QQ}|_a \geq \frac{|\tilde{\QQ}|_a}{2},
\]
and
\[
\|F_{\vep_j}\|_{L^{p,a}(\QQ_2)} + \|f_{\vep_j}\|_{L_\infty^q(Q_2)} \leq 1,
\]
for every $j \in \NN$. Then there exists $\sigma > 0$ depending only on $N$, $a$, $p$, $q$ and $\dd$ such that for every $U$ and every $\vep_{j_k} \to 0$ satisfying \eqref{eq:ConvergenceOscDecay} (with $j = j_k$ and $r=2$), we have
\[
|\{0 < U_{\vep_{j_k}} < 1/2 \} \cap (\QQ_1 \cup \tilde{\QQ})|_a \geq \sigma,
\]
for every $k \in \NN$.
\end{lem}
\begin{proof} We assume by contradiction there exist sequences $\vep_k := \vep_{j_k} \to 0$, $F_k:= F_{\vep_{j_k}}$, $f_k:=f_{\vep_{j_k}}$ satisfying 
\begin{equation}\label{eq:UnifBoundsFjfj}
\|F_k\|_{L^{p,a}(\QQ_2)} + \|f_k \|_{L_\infty^q(Q_2)} \leq 1,
\end{equation}
$U \in C([-4,4]:L^{2,a}(\BB_2))$ and a sequence of weak solutions $U_k := U_{\vep_{j_k}}$ satisfying $U_k \leq 1$ and $U_k \to U$ in $C([-4,4]:L^{2,a}(\BB_2))$, such that if $A_k := \{U_k \geq 1/2 \} \cap \QQ_1$, $C_k := \{U_k \leq 0 \} \cap \tilde{\QQ}$ and $D_k := \{0 < U_k < 1/2 \} \cap (\QQ_1 \cup \tilde{\QQ})$, then
\begin{equation}\label{eq:IsopAbsurdCond}
\begin{aligned}
|A_k|_a \geq \dd, \qquad |C_k|_a \geq |\tilde{\QQ}|_a/2, \qquad |D_k|_a \leq \tfrac{1}{k},
\end{aligned}
\end{equation}
for every $k \in \NN$.

Let us set $V_k := (U_k)_+$ and $V := U_+$. By assumption, we have $V_k \to V$ in $C([-4,4]:L^{2,a}(\BB_2))$. Further, by Lemma \ref{lem:IneqTruncations}, each $V_k$ is a nonnegative weak subsolution and thus, combining the assumption $U_k \leq 1$ and \eqref{eq:EnergyEstimate}, we may assume that $V_k \rightharpoonup V$ in $L^2(-2,2:H^{1,a}(\BB_1))$ which, in turn, implies 
\[
\|\nabla V\|_{L^{2,a}((-2,2)\times\BB_1))} \leq C,
\]
for some $C > 0$ depending only on $N$, $a$, $p$ and $q$. Notice that by $C([-4,4]:L^{2,a}(\BB_2))$ convergence, we also have convergence in measure, that is, for every $\ep > 0$,
\begin{equation}\label{eq:IsopConvMeasure}
\lim_{k \to +\infty} |\{ |V_k - V| \geq \ep \} \cap \QQ_2|_a = 0.
\end{equation}

\emph{Step 1.} We prove that
\begin{equation}\label{eq:IsopFirstMeasureLimit}
|\{0 < V < \tfrac12\} \cap (\QQ_1 \cup \tilde{\QQ})|_a = 0.
\end{equation}
To see this, we notice that given $\ep > 0$, the set $\{\ep \leq V \leq \tfrac12 -\ep\}$ can be trivially written as the disjoint union of $\{|V_k - V| \geq \ep\}$ and $\{|V_k - V| < \ep\}$. Furthermore,
\[
\{\ep \leq V \leq \tfrac12 -\ep\} \cap \{|V_k - V| < \ep\} \subseteq \{\ep \leq V \leq \tfrac12 -\ep\} \cap \{0 < V_k <  \tfrac12\},
\]
which implies
\[
|\{\ep \leq V \leq \tfrac12 -\ep\} \cap (\QQ_1 \cup \tilde{\QQ})|_a \leq |\{|V_k - V| \geq \ep\} \cap (\QQ_1 \cup \tilde{\QQ})|_a + |\{ 0 < V_k < \tfrac12 \} \cap (\QQ_1 \cup \tilde{\QQ})|_a.
\]
Since $\QQ_1 \cup \tilde{\QQ} \subset \QQ_2$, we may pass to the limit as $k \to +\infty$ and notice that the last relation in \eqref{eq:IsopAbsurdCond} and \eqref{eq:IsopConvMeasure} yield $|\{\ep \leq V \leq \tfrac12 -\ep\} \cap (\QQ_1 \cup \tilde{\QQ})|_a = 0$ and thus, by the arbitrariness of $\ep > 0$, \eqref{eq:IsopFirstMeasureLimit} follows.

\smallskip

\emph{Step 2.} Now we show that
\begin{equation}\label{eq:Isop2fundIneqLimit}
|\{ V = 0 \} \cap \tilde{\QQ}|_a \geq \frac{|\tilde{\QQ}|_a}{2}.
\end{equation}
First, we notice that \eqref{eq:IsopFirstMeasureLimit}, combined with the fact that $\nabla V(\cdot,t) \in L^{2,a}(\BB_1)$ for a.e. $t \in (-2,1)$, allows us to apply Lemma \ref{lem:IsoperimetricIneqElliptic}, obtaining
\begin{equation}\label{eq:IsopSignUaet}
\begin{aligned}
\text{either } \; V(\cdot,t) = 0 \quad &\text{ in } \BB_1 \\
\text{or } \; V(\cdot,t) \geq \tfrac12 \quad &\text{ in } \BB_1,
\end{aligned}
\end{equation}
for a.e. $t \in (-2,1)$. Now, following the proof of \eqref{eq:IsopFirstMeasureLimit}, we notice that if $V_k \leq 0$ and $\ep > 0$ is fixed, then either $|V_k - V| \geq \ep$ or $V \leq \ep$. Therefore
\[
\frac{|\tilde{\QQ}|_a}{2} \leq |\{V_k = 0 \} \cap \tilde{\QQ}|_a \leq |\{|V_k - V| \geq \ep\} \cap \tilde{\QQ}|_a + |\{  V < \ep \} \cap \tilde{\QQ}|_a,
\]
and thus, taking the limit as $k \to + \infty$ and then as $\ep \to 0$, we deduce \eqref{eq:Isop2fundIneqLimit}, where we have exploited \eqref{eq:IsopAbsurdCond} and \eqref{eq:IsopConvMeasure} again.

\smallskip

\emph{Step 3.} This is the most delicate part of the proof. We prove that
\begin{equation}\label{eq:IsopFundIneqLimit}
V = 0 \quad \text{ in } \QQ_1.
\end{equation}
We test the differential inequality of $V_k$ with $\eta_k = V_k \psi_k^2$, for a sequence of test functions $\psi_k$ we will choose in a moment. Following the proof of Lemma \ref{lem:EnergyEstimate}, we obtain (cf. \eqref{eq:L2LinfFirstDiffInequality})
\[
\begin{aligned}
\tfrac{1}{2} &\int_{\mathbb{Q}_2} |y|^a \partial_t(V_k^2) \psi_k^2 \dX\rd t +  \int_{\mathbb{Q}_2} |y|^a |\nabla V_k|^2 \psi_k^2 \dX\rd t + \varepsilon_k \int_{\mathbb{Q}_2} |y|^a |\partial_t V_k|^2 \psi_k^2 \dX\rd t \\
&\leq - 2\int_{\mathbb{Q}_2} |y|^a ( \psi_k \nabla V_k ) \cdot ( V_k \nabla \psi_k) \dX\rd t - \vep_k \int_{\mathbb{Q}_2} |y|^a \partial_t V_k ( V_k \partial_t(\psi_k^2)) \dX\rd t \\
&\quad+ \int_{\mathbb{Q}_2} |y|^a (F_k)_+ V_k \psi_k^2\dX\rd t + \int_{Q_2} (f_k)_+ v_k \psi_k^2|_{y=0} \dX\rd t.
\end{aligned}
\]
Now, using Young's inequality, we reabsorb the gradient part of the first term in the r.h.s. with the second term in the l.h.s., and we apply H\"older's inequality to the last three terms in the r.h.s. to obtain
\begin{equation}\label{eq:EnIneq1k}
\begin{aligned}
\tfrac{1}{2} \int_{\mathbb{Q}_2} |y|^a \partial_t(V_k^2) \psi_k^2 \dX\rd t &\leq C_0 \int_{\mathbb{Q}_2} |y|^a V_k^2 |\nabla \psi_k|^2 \dX\rd t + \vep_k \|\partial_t V_k\|_{L^{2,a}(\mathbb{Q}_2)} \|V_k \partial_t(\psi_k^2)\|_{L^{2,a}(\mathbb{Q}_2)}  \\
&\quad+ \|F_k\|_{L^{p,a}(\QQ_2)} \| V_k \psi_k^2\|_{L^{p',a}(\QQ_2)} + \|f_k\|_{L_\infty^q(Q_2)} \| v_k \psi_k|_{y=0}^2\|_{L_1^{q'}(Q_2)} \\
&\leq C_0 \int_{\mathbb{Q}_2} |y|^a |\nabla \psi_k|^2 \dX\rd t + \vep_k \|\partial_t V_k\|_{L^{2,a}(\mathbb{Q}_2)} \|\partial_t(\psi_k^2)\|_{L^{2,a}(\mathbb{Q}_2)} \\
&\quad+ \|\psi_k^2\|_{L^{p',a}(\QQ_2)} + \|\psi_k|_{y=0}^2\|_{L_1^{q'}(Q_2)},
\end{aligned}
\end{equation}
for some numerical constant $C_0 > 0$ (the last inequality easily follows in light of \eqref{eq:UnifBoundsFjfj} and that $0 \leq V_k \leq 1$).

Now, we fix $-2 < s < \tau < 1$ and choose $\psi_k^2(X,t) = \varphi^2(X)\chi_k(t)$ where $\varphi \in C_0^\infty(\BB_1)$ satisfies
\begin{equation}\label{eq:IsopAssVarphi}
\varphi \geq 0, \qquad \int_{\BB_1} |y|^a \varphi^2(X) \dX = 1,
\end{equation}
while $\chi_k$ is a Lipschitz approximation of $\chi_{[s,\tau]}$ as $k \to +\infty$, defined as follows
\[
\chi_k(t):=
\begin{cases}
0 &\quad t \leq s - \dd_k \text{ or } t \geq \tau + \dd_k \\
\tfrac{1}{\dd_k}(t - s + \dd_k) &\quad t \in (s-\dd_k,s) \\
\tfrac{1}{\dd_k}(\tau + \dd_k - t) &\quad t \in (\tau,\tau + \dd_k) \\
1 &\quad t \in [s,\tau],
\end{cases}
\]
for some positive $\dd_k \to 0$ as $k \to +\infty$. Notice that, by dominated convergence, 
\begin{equation}\label{eq:DeGioGradpsi}
\int_{\mathbb{Q}_2} |y|^a |\nabla \psi_k|^2 \dX\rd t = \int_{-2}^2 \int_{\mathbb{B}_1} |y|^a |\nabla \varphi|^2 \chi_k(t) \dX\rd t \to \int_s^\tau \int_{\mathbb{B}_1} |y|^a |\nabla \varphi|^2 \dX\rd t \leq C(\tau -s), 
\end{equation}
as $k \to +\infty$, for some $C > 0$ depending only on $N$ and $a$. Similar
\begin{equation}\label{eq:DeGioLpLqpsi}
\begin{aligned}
\|\psi_k^2\|_{L^{p',a}(\QQ_2)} + \|\psi_k|_{y=0}^2\|_{L_1^{q'}(Q_2)} &\to \Big( \int_s^\tau \int_{\mathbb{B}_1} |y|^a \varphi^{2p'} \dX\rd t \Big)^{1/p'} +  \int_s^\tau  \Big( \int_{B_1} \varphi|_{y=0}^{2q'} \dx \Big)^{1/q'} \rd t  \\
&\leq C \big[ (\tau-s)^{1/p'} + (\tau-s) \big] \leq C(\tau-s)^{1/p'},
\end{aligned}
\end{equation}
as $k \to + \infty$, for some new $C > 0$ depending also on $p$ and $q$. Furthermore, by taking $C > 0$ larger and using \eqref{eq:EnergyEstimate}, we have
\begin{equation}\label{eq:DeGioVanishpart}
\begin{aligned}
\vep_k \|\partial_t V_k\|_{L^{2,a}(\mathbb{Q}_2)} \|\partial_t(\psi_k^2)\|_{L^{2,a}(\mathbb{Q}_2)} &\leq C \vep_k^{1/2} \|\partial_t(\psi_k^2)\|_{L^{2,a}(\mathbb{Q}_2)} = C \vep_k^{1/2} \Big( \int_{\mathbb{Q}_2} |y|^a \varphi^4(X) (\chi_k'(t))^2 \dX\rd t \Big)^{1/2} \\
& \leq C \vep_k^{1/2} \|\varphi^2\|_{L^{2,a}(\BB_2)} \Big( \tfrac{1}{\dd_k^2} \int_{s-\dd_k}^s \rd t + \tfrac{1}{\dd_k^2} \int_\tau^{\tau+\dd_k} \rd t \Big)^{1/2} \leq C \sqrt{\frac{\vep_k}{\dd_k}} \to 0,
\end{aligned}
\end{equation}
by choosing $\dd_k = \vep_k^{1/2}$. Finally, using the definition of $\chi_k$ and integrating by parts in time, we find
\begin{equation}\label{eq:DeGioDeriv}
\begin{aligned}
\tfrac{1}{2} \int_{\mathbb{Q}_2} |y|^a \partial_t(V_k^2) \psi_k^2 \dX\rd t &= \tfrac{1}{2} \int_{\mathbb{Q}_2} |y|^a \partial_t(V_k^2) \varphi^2(X) \chi_k(t) \dX\rd t \\
&= \tfrac{1}{2\dd_k} \int_\tau^{\tau+\dd_k}\int_{\mathbb{B}_1} |y|^a V_k^2 \varphi^2(X) \dX\rd t - \tfrac{1}{2\dd_k} \int_{s-\dd_k}^s\int_{\mathbb{B}_1} |y|^a V_k^2 \varphi^2(X) \dX\rd t \\
&\to \tfrac{1}{2} \int_{\mathbb{B}_1} |y|^a V^2(\tau) \varphi^2(X) \dX - \tfrac{1}{2} \int_{\mathbb{B}_1} |y|^a V^2(s) \varphi^2(X) \dX,
\end{aligned}
\end{equation}
as $k \to + \infty$. To check the limit, we use that $V_k$, $V$ and $\varphi$ are bounded in $\QQ_2$, and so
\[
\begin{aligned}
\tfrac{1}{\dd_k} \int_\tau^{\tau+\dd_k}\int_{\mathbb{B}_1} |y|^a |V_k^2 - V^2| \varphi^2(X) \dX\rd t &\leq \tfrac{C}{\dd_k} \int_\tau^{\tau+\dd_k}\int_{\mathbb{B}_1} |y|^a |V_k - V| \dX\rd t \\
&\leq \tfrac{C}{\dd_k} \int_\tau^{\tau+\dd_k} \|V_k(t) - V(t)\|_{L^{2,a}(\BB_1)} \rd t \\
&\leq C \sup_{t \in (-4,4)} \|V_k(t) - V(t)\|_{L^{2,a}(\BB_1)} \to 0,
\end{aligned}
\]
as $k \to +\infty$ and thus the limit in \eqref{eq:DeGioDeriv} follows by triangular inequality. Consequently, passing to the limit as $k \to +\infty$ in \eqref{eq:EnIneq1k} and using \eqref{eq:DeGioGradpsi}, \eqref{eq:DeGioLpLqpsi}, \eqref{eq:DeGioVanishpart} and \eqref{eq:DeGioDeriv}, we deduce
\[
\tfrac{1}{2} \int_{\mathbb{B}_1} |y|^a V^2(\tau) \varphi^2(X) \dX - \tfrac{1}{2} \int_{\mathbb{B}_1} |y|^a V^2(s) \varphi^2(X) \dX \leq C(\tau - s)^{1/p'},
\]
for some constant $C > 0$ depending only $N$, $a$, $p$, $q$ and $\varphi$.

Now, by \eqref{eq:Isop2fundIneqLimit}, we may choose $s \in (-2,-1)$ such that $V(\cdot,s) = 0$ in $\BB_1$, to obtain
\[
\int_{\BB_1} |y|^a V^2(X,\tau)\varphi^2(X)\dX \leq C (\tau - s)^{1/p'}.
\]
On the other hand, by \eqref{eq:IsopSignUaet}, we know that either $V(\cdot,\tau) = 0$ or $V(\cdot,\tau) \geq \tfrac12$ in $\BB_1$ for a.e. $\tau \in (-2,1)$. However, if $V(\cdot,\tau) \geq \tfrac12$ in $\BB_1$, the above inequality and \eqref{eq:IsopAssVarphi} yield
\[
(\tau - s)^{1/p'} \geq c_0,
\]
for some $c_0 > 0$ depending only $N$, $a$, $p$, $q$ and $\varphi$, and thus it must be $V(\cdot,\tau) = 0$ in $\BB_1$ for a.e. $\tau \leq s + c_0^{1/p'}$. Iterating this procedure, \eqref{eq:IsopFundIneqLimit} follows.

\smallskip

\emph{Step 4.} Finally, arguing as in \eqref{eq:IsopFirstMeasureLimit}, whenever $V_k \geq \tfrac12$, then either $|V-V_k| \geq \ep$ or $V > \tfrac12 - \ep$ and so
\[
\begin{aligned}
\dd \leq |\{V_k \geq 1/2 \} \cap \QQ_1|_a \leq |\{|V-V_k| \geq \ep \} \cap \QQ_1|_a + |\{V \geq 1/2 - \ep \} \cap \QQ_1|_a.
\end{aligned}
\]
Passing to the limit as $k \to + \infty$ and then as $\ep \to 0$, we deduce $|\{V \geq 1/2 \} \cap \QQ_1|_a \geq \dd$, which is in contradiction with \eqref{eq:IsopFundIneqLimit}.
\end{proof}
\begin{rem}\label{rem:CompactnessPar}
To understand the role played by the assumption \eqref{eq:ConvergenceOscDecay}, it is useful to compare with the classical parabolic framework. The main difference here is that the parabolic equation (combined with the parabolic energy estimate) gives enough compactness for carrying out the contradiction argument. Indeed, let $U_k$ be a sequence of weak solutions to \eqref{eq:ParabolicEqRem} satisfying $\|U_k\|_{L^{2,a}(\QQ_2)} \leq 1$. For every $k \in \NN$, we have 
\[
\int_{\mathbb{Q}_1} |y|^a \partial_t U_k \eta \dX\rd t = - \int_{\mathbb{Q}_1} |y|^a \nabla U_k \cdot\nabla\eta \dX\rd t,
\]
for every $\eta \in L^2(-1,1:H_0^{1,a}(\BB_1))$. Thus, noticing that the sequence $\{U_k\}_k$ is uniformly bounded in $L^2(-1,1:H^{1,a}(\BB_1))$ (by the parabolic energy estimate, see Remark \ref{rem:ParEE}) and using the equation of $U_k$ above, we obtain that $\{\partial_tU_k\}_k$ is uniformly bounded in $L^2(-1,1:H^{-1,a}(\BB_1))$ and so, by the Aubin-Lions lemma, we have $U_k \to U$ in $L^{2,a}(\QQ_1)$, up to passing to a suitable subsequence. This is enough to show that $U$ satisfies \eqref{eq:IsopFirstMeasureLimit} and \eqref{eq:Isop2fundIneqLimit}. To prove \eqref{eq:IsopFundIneqLimit}, it suffices to notice that, since $\{U_k(t)\}_k$ is uniformly bounded in $H^{1,a}(\BB_1)$ for a.e. $t \in (-1,1)$, we may also assume $U_k(t_n) \to U(t_n)$ in $L^{2,a}(\BB_1)$ for a finite increasing sequence of times $t_n \in (-1,1)$, up to passing to another subsequence. This allows us to pass to the limit into \eqref{eq:EnEs1HeatEq} (with $\tau = t_{n+1}$ and $s = t_n$) and complete the argument of Step 3.

On the contrary, in our ``approximating setting'', the weak formulation \eqref{eq:WeakEquation} (with $F_\vep = 0$ and $ f_\vep =0$) is
\[
\int_{\mathbb{Q}_1} |y|^a ( \varepsilon_k \partial_t U_{\vep_k} \partial_t \eta + \partial_t U_{\vep_k} \eta + \nabla U_{\vep_k} \cdot\nabla\eta ) \dX\rd t = 0,
\]
for every $\eta \in C_0^\infty(\QQ_1)$, and the energy estimate \eqref{eq:EnergyEstimate} gives uniform bounds for $\{U_{\vep_k}\}_{k\in\NN}$ in 
\[
L^2(-1,1:H^{1,a}(\BB_1)) \cap L^\infty(-1,1:L^{2,a}(\BB_1)),
\]
for sequences of solutions uniformly bounded in $L^{2,a}(\QQ_2)$. It is not clear to us if these ingredients can be combined to obtain strong $L^{2,a}(\QQ_1)$ compactness, or not. For this reason we require $C([-1,1]:L^{2,a}(\BB_1))$ strong compactness in \eqref{eq:ConvergenceOscDecay} which, as already mentioned, is guaranteed for families of minimizers of \eqref{eq:Functional} by Proposition \ref{prop:weaklimit}.
\end{rem}
\begin{lem}\label{lem:OscDecay}
Let $N \geq 1$, $a \in (-1,1)$, $(p,q)$ satisfying \eqref{eq:AssumptionsExponents}, and $\tilde{\QQ} := \BB_1\times(-2,-1)$. There exist $\dd_0,\tilde{\theta}_0 \in (0,1)$ depending only on $N$, $a$, $p$ and $q$ such that for every $U$, every sequence $\vep_j \to 0$ and every sequence $\{U_{\vep_j}\}_{j \in \NN}$ of weak solutions in $\QQ_2$ satisfying \eqref{eq:ConvergenceOscDecay} with $r=2$, such that
\[
U_{\vep_j} \leq 1, \qquad |\{U_{\vep_j} \leq 0 \} \cap \tilde{\QQ}|_a \geq \frac{|\tilde{\QQ}|_a}{2},
\]
and
\begin{equation}\label{eq:AssDd0Ff}
\|F_{\vep_j}\|_{L^{p,a}(\QQ_2)} + \|f_{\vep_j}\|_{L_\infty^q(Q_2)} \leq \dd_0,
\end{equation}
for every $j \in \NN$, then
\[
U_{\vep_j} \leq 1 - \tilde{\theta}_0 \quad \text{ in } \QQ_{1/2},
\]
for every $j \in \NN$.
\end{lem}
\begin{proof} Let $\sigma > 0$ be as in Lemma \ref{lem:IsoperimetricIneq}, $\dd_0 := 2^{-n_0}\sqrt{\delta}$, where $n_0$ is the largest integer such that 
\[
(n_0-1)\sigma \leq |\QQ_1 \cup \tilde{\QQ}|_a,
\]
and $\delta \in (0,1)$ will be chosen later. Let us set $U_j = U_{\vep_j}$, $F_j = F_{\vep_j}$, $f_j = f_{\vep_j}$ and define
\[
V_{j,n} := 2^n \big[ U_j - (1 - 2^{-n}) \big], \quad j,n \in \NN.
\]
The assumptions on $U_j$ imply that $V_{j,n}$ is a solution in $\QQ_2$ with $F_{j,n} := 2^n F_j$ and $f_{j,n} := 2^n f_j$, satisfying 
\[
V_{j,n} \leq 1, \qquad |\{V_{j,n} \leq 0 \} \cap \tilde{\QQ}|_a \geq \frac{|\tilde{\QQ}|_a}{2},
\]
for every $j,n \in \NN$ and $V_{j,n} \to V_n := 2^n[U - (1 - 2^{-n})]$ in $C([-4,4]:L^{2,a}(\BB_2))$ as $j \to +\infty$. Further, notice that by \eqref{eq:AssDd0Ff} and the definition of $\dd_0$, we have
\[
\|F_{j,n}\|_{L^{p,a}(\QQ_2)} + \|f_{j,n}\|_{L_\infty^q(Q_2)} \leq 2^n\dd_0 \leq 2^{n-n_0} \sqrt{\dd} \leq 1,
\]
for every $n \leq n_0$ and every $j$. Now, fix $\dd \in (0,1)$ and assume
\begin{equation}\label{eq:OscDecayL2NormCond}
\int_{\QQ_1} |y|^a (V_{j,n+1})_+^2 \dX \rd t \geq \dd.
\end{equation}
Then, using the definition of $V_{j,n}$ (and $V_{j,n+1}$) and that $V_{j,n} \leq 1$, we easily see that
\[
|\{V_{j,n} \geq \tfrac12 \}\cap \QQ_1|_a = |\{V_{j,n+1} \geq 0 \}\cap \QQ_1|_a \geq \int_{\QQ_1} |y|^a (V_{j,n+1})_+^2 \dX \rd t \geq \dd.
\]
Consequently, if \eqref{eq:OscDecayL2NormCond} holds true for every $n \leq n_0$ and some $j \in \NN$, $V_{j,n}$ satisfies the assumptions of Lemma \ref{lem:IsoperimetricIneq}, and so
\[
|\{0 < V_{j,n} < 1/2 \} \cap (\QQ_1 \cup \tilde{\QQ})|_a \geq \sigma,
\]
for every $n \leq n_0$. However, since $\{0 < V_{j,n} < 1/2 \} \cap \{0 < V_{j,m} < 1/2 \} = \emptyset$ for every $m \not= n$ and every $j$, the above inequality implies $|\QQ_1 \cup \tilde{\QQ}|_a \geq n_0\sigma$, in contradiction with the definition of $n_0$. Consequently, \eqref{eq:OscDecayL2NormCond} fails for $n=n_0$ and every $j \in \NN$, that is
\[
\int_{\QQ_1} |y|^a (V_{j,n_0+1})_+^2 \dX \rd t \leq \dd,
\]
for every $j \in \NN$. Let us set $k_0 := n_0 + 1$. Since $(V_{j,k_0})_+$ is a nonnegative subsolution (with $(F_{j,k_0})_+$ and $(f_{j,k_0})_+$), we obtain by \eqref{eq:L2LinfEstimate} and the definition of $\dd_0$
\[
\| (V_{j,k_0})_+\|_{L^\infty(\mathbb{Q}_{1/2})} \leq C \big( \| (V_{j,k_0})_+ \|_{L^{2,a}(\QQ_1)} + \|(F_{j,k_0})_+\|_{L^{p,a}(\QQ_1)} + \|(f_{j,k_0})_+\|_{L_\infty^q(Q_1)} \big) \leq C\sqrt{\dd},
\]
for some $C > 0$ depending only on $N$, $a$, $p$ and $q$, and every $j \in \NN$. Now, taking $\dd$ such that $C\dd \leq \tfrac12$, the above inequality gives $V_{j,k_0} \leq \tfrac12$ in $\QQ_{1/2}$ for every $j$ and thus
\[
U_j \leq 2^{-k_0-1} + \big( 1 - 2^{-k_0} \big) = 1 - 2^{-k_0-1} \quad \text{ in } \QQ_{1/2},
\]
for every $j$, which is our thesis, choosing $\tilde{\theta}_0 := 2^{-k_0-1}$.
\end{proof}
\begin{cor}\label{thm:OscDecay}
Let $N \geq 1$, $a \in (-1,1)$ and $(p,q)$ satisfying \eqref{eq:AssumptionsExponents}. There exist $\dd_0,\theta_0 \in (0,1)$ depending only on $N$, $a$, $p$ and $q$ such that for every $U$, every sequence $\vep_j \to 0$ and every sequence $\{U_{\vep_j}\}_{j \in \NN}$ of weak solutions in $\QQ_4$ satisfying \eqref{eq:ConvergenceOscDecay} with $r=4$ and \eqref{eq:FinalBoundFjfj}, there exist a subsequence $j_k \to +\infty$ such that 
\[
\begin{aligned}
\osc_{\QQ_{1/2}} U_{\vep_{j_k}} \leq (1 - \theta_0) \osc_{\QQ_2} U_{\vep_{j_k}} + \tfrac{1}{\dd_0}\big( \|F_{\vep_{j_k}}\|_{L^{p,a}(\QQ_2)} + \|f_{\vep_{j_k}}\|_{L_\infty^q(Q_2)} \big) 
\end{aligned}
\]
for every $k \in \NN$. 
\end{cor}
\begin{proof} Let $\dd_0,\tilde{\theta}_0 >0$ be as in Lemma \ref{lem:OscDecay}, and let $U_j = U_{\vep_{j}}$, $F_j = F_{\vep_{j}}$ and $f_j = f_{\vep_{j}}$. We fix $\delta > 0$, and define
\[
V_j := \frac{2}{K_j} \left( U_j - \frac{\esssup_{\QQ_2} U_j + \essinf_{\QQ_2} U_j}{2} \right), \qquad K_j : =  \delta + \osc_{\QQ_2} U_j + \tfrac{2}{\dd_0}(\|F_j\|_{L^{p,a}(\QQ_2)} + \|f_j\|_{L_\infty^q(Q_2)}).
\]
We have $K_j \geq \dd$ for every $j \in \NN$ and, further, each $V_j$ is a weak solution in $\QQ_4$ (and thus in $\QQ_2$) with $\bar{F}_j := \tfrac{2}{K_j} F_j$, $\bar{f}_j := \tfrac{2}{K_j} f_j$ satisfying
\[
V_j \leq 1 \quad\text{in } \QQ_2, \qquad \|\bar{F}_j\|_{L^{p,a}(\QQ_2)} + \|\bar{f}_j\|_{L_\infty^q(Q_2)} = \frac{2}{K_j} (\|F_j\|_{L^{p,a}(\QQ_2)} + \|f_j\|_{L_\infty^q(Q_2)}) \leq \dd_0,
\]
for every $j$, by definition of $K_j$.

Now, in view of \eqref{eq:ConvergenceOscDecay}, we have that $\{U_j\}_{j\in\NN}$ is uniformly bounded in $L^{2,a}(\QQ_4)$ and thus, by \eqref{eq:L2LinfEstimate} and \eqref{eq:FinalBoundFjfj}, it is bounded in $L^{\infty}(\QQ_2)$. This implies the existence of a subsequence $j_k \to + \infty$, $K \in [\dd,+\infty)$ and $l \in \RR$ (both $K$ and $l$ are finite depending on $\dd_0$ and $\|U\|_{L^{2,a}(\QQ_2)}$) such that
\[
K_k \to K, \qquad \esssup_{\QQ_2} U_k + \essinf_{\QQ_2} U_k \to l
\]
as $k \to + \infty$, where $U_k := U_{j_k}$ and $K_k := K_{j_k}$. As a consequence, one can easily verify that
\[
V_k \to V := \frac{2}{K} \left( U - \frac{l}{2} \right) \quad \text{ in } C([-4,4]:L^{2,a}(\BB_2)),
\]
as $k \to +\infty$, where $V_k := V_{j_k}$. Furthermore, we notice that
\[
\begin{aligned}
\text{if } \qquad |\{V_k \leq 0 \} \cap \tilde{\QQ}|_a &\geq \tfrac{|\tilde{\QQ}|_a}{2} \quad \text{for a finite number of indexes} \\ 
\text{then } \quad |\{-V_k \leq 0 \} \cap \tilde{\QQ}|_a &\geq \tfrac{|\tilde{\QQ}|_a}{2} \quad \text{for an infinite number of indexes},
\end{aligned}
\]
and thus, up to passing to an additional subsequence and eventually considering $-V_k$ instead of $V_k$, we may assume
\[
|\{V_k \leq 0 \} \cap \tilde{\QQ}|_a \geq \frac{|\tilde{\QQ}|_a}{2},
\]
for every $k \in \NN$. Then, Lemma \ref{lem:OscDecay} yields
\[
V_k \leq 1-\tilde{\theta}_0 \quad \text{ in } \QQ_{1/2},
\]
for every $k\in\NN$, that is
\[
U_k \leq \tfrac{1-\tilde{\theta}_0}{2} \big( \delta + \osc_{\QQ_2} U_k \big) + \frac{\esssup_{\QQ_2} U_k + \essinf_{\QQ_2} U_k}{2} + \tfrac{1-\tilde{\theta}_0}{\dd_0}(\|F_k\|_{L^{p,a}(\QQ_2)} + \|f_k\|_{L_\infty^q(Q_2)}) \quad \text{ in } \QQ_{1/2}.
\]
Taking the $\sup_{\QQ_{1/2}}$ and subtracting $\inf_{\QQ_{1/2}} U_k$ in both sides, it is not difficult obtain
\[
\osc_{\QQ_{1/2}} U_k \leq \tfrac{1-\tilde{\theta}_0}{2}  \delta + \big( 1 - \tfrac{\tilde{\theta}_0}{2} \big) \osc_{\QQ_2} U_k + \tfrac{1-\tilde{\theta}_0}{\dd_0}(\|F_k\|_{L^{p,a}(\QQ_2)} + \|f_k\|_{L_\infty^q(Q_2)}),
\]
for every $k \in \NN$. Our thesis follows by passing to the limit as $\delta \to 0$ and choosing $\theta_0 = \tfrac{\tilde{\theta}_0}{2}$.
\end{proof}
\begin{proof}[Proof of Proposition \ref{thm:CalphaBound}] Let $\dd_0$ and $\theta_0$ as in Corollary \ref{thm:OscDecay}, and $\delta > 0$. We set $U_j := U_{\vep_j}$, $F_j := F_{\vep_j}$, $f_j := f_{\vep_j}$ and define
\[
V_j(X,t) := \frac{1}{K_j} U_j(X,t), \qquad K_j := \delta + \| U_j \|_{L^\infty(\QQ_4)} + \|F_j \|_{L^{p,a}(\QQ_4)} + \|f_j \|_{L_\infty^q(Q_4)},
\]
for every $j \in \NN$. Notice that thanks to \eqref{eq:FinalBoundFjfj}, \eqref{eq:ConvergenceOscDecay} and \eqref{eq:L2LinfEstimate}, $\{K_j\}_{j\in\NN}$ is uniformly bounded. Thus, similar to the proof above, we have $K_j \to K \geq \delta$, as $j \to +\infty$,  up to passing to a subsequence. Consequently,
\begin{equation}\label{eq:ConvCalphaBoundsV}
V_j \to V \quad \text{ in } C([-16,16]:L^{2,a}(\BB_4)), \qquad V = \frac{1}{K}U,
\end{equation}
as $j\to + \infty$.

Now, let $r_n := 4^{-n}$, $n \in \NN$. We show that there exist $n_0 \in \NN$, $C > 0$ and $\alpha \in (0,1)$ depending only on $N$, $a$, $p$ and $q$ such that, up to passing to a subsequence $\vep_k := \vep_{j_k} \to 0$, if $V_k := V_{j_k}$, then for every $(X_0,t_0) := (x_0,y_0,t_0) \in \QQ_1$ we have
\begin{equation}\label{eq:OscDecayFinal}
\osc_{\QQ_{4^{-n+1}}(X_0,t_0)} V_k \leq C 4^{-\alpha n}, 
\end{equation}
for every $n \geq n_0$ and every $k$ such that $\vep_k \leq r_n^2$. A standard contradiction argument (see for instance \cite[Proof of Theorem 6.1]{AudritoTerracini2020:art}, \cite[Proof of Theorem 4.1]{SireTerVita1:art} and \cite[Proof of Theorem 4.2]{SireTerVita2:art}) shows that it is enough to prove \eqref{eq:OscDecayFinal} for points $(X_0,t_0)$ with $y_0 = 0$: basically, if H\"older regularity (or oscillation decay) fails at some point, then such point must belong to the region where the weight $|y|^a$ is degenerate or singular). As a consequence, since the equation of $V_k$ is invariant under translations w.r.t. $x$ and $t$, it is enough to consider the case $(X_0,t_0) = 0$.

With this goal in mind, let us define
\[
V_{j,n}(X,t) := \frac{1}{K_j} V_j(r_nX,r_n^2t), \qquad j,n \in \NN.
\]
By Remark \ref{rem:ScalingProp} each $V_{j,n}$ satisfies
\[
\int_{\mathbb{Q}_4} |y|^a ( \tfrac{\varepsilon_j}{r_n^2} \partial_t V_{j,n} \partial_t \eta + \partial_t V_{j,n} \eta + \nabla V_{j,n} \cdot\nabla\eta - F_{j,n} \eta) \dX\rd t - \int_{Q_4} f_{j,n} \eta|_{y=0} \dx \rd t = 0,
\]
for every $\eta \in C_0^\infty(\mathbb{Q}_4)$, where
\[
\begin{aligned}
F_{j,n}(X,t) &:= \frac{1}{K_j} r_n^2 F_j(r_nX,r_n^2t),  \\
f_{j,n}(x,t) &:= \frac{1}{K_j} r_n^{1-a} f_j(r_nx,r_n^2t).
\end{aligned}
\]
By definition and scaling, we easily see that $\|V_{j,n}\|_{L^\infty(\QQ_4)} \leq 1$ and
\[
\|F_{j,n}\|_{L^{p,a}(\QQ_4)} + \|f_{j,n}\|_{L_\infty^q(Q_4)} \leq r_n^{2 - \frac{N+3+a}{p}} \frac{\|F_j\|_{L^{p,a}(\QQ_4)}}{K_j} + r_n^{1 - a - \frac{N}{q}} \frac{\|f_j\|_{L_\infty^q(Q_4)}}{K_j} \leq r_n^\nu,
\]
for every $j,n \in \NN$, where we have set 
\[
\nu := \min\{2 - \tfrac{N+3+a}{p},1 - a - \tfrac{N}{q}\} > 0.
\]
The positivity of $\nu$ follows from \eqref{eq:AssumptionsExponents}. In a moment we will choose $n_0 \in \NN$ such that $r_n^\nu \leq \dd_0$, for every $n \geq n_0$ and thus we may assume
\[
\|F_{j,n}\|_{L^{p,a}(\QQ_4)} + \|f_{j,n}\|_{L_\infty^q(Q_4)} \leq \delta_0,
\]
for every $n \geq n_0$ and $j \in \NN$. Further, exploiting \eqref{eq:ConvCalphaBoundsV}, it is not difficult to check that for every fixed $n \in \NN$, we have
\[
V_{j,n} \to \tilde{V}_n \quad \text{ in } C([-16,16]:L^{2,a}(\BB_4)),
\]
as $j \to +\infty$, where $\tilde{V}_n(X,t) := V(r_nX,r_n^2t)$.

Then, for every fixed $n \geq n_0$ the sequence $\{V_{j,n}\}_{j\in\NN}$ satisfies the assumptions of Corollary \ref{thm:OscDecay} and thus there exist subsequences $j_k \to +\infty$ and $\vep_k := \vep_{j_k} \to 0$ such that for every fixed $n \geq n_0$ and every $k$ such that $\vep_k \leq r_n^2$, we have
\[
\osc_{\QQ_1} V_{k,n} \leq (1 - \theta_0) \osc_{\QQ_4} V_{k,n} + \tfrac{1}{\dd_0} r_n^\nu,
\]
where $V_{k,n} := V_{j_k,n}$. Re-writing such inequality in terms of $V_k$, it follows
\begin{equation}\label{eq:OscDecayCalpha}
\osc_{\QQ_{4^{-n}}} V_k \leq (1 - \theta_0) \osc_{\QQ_{4^{-n+1}}} V_k + \tfrac{1}{\dd_0} 4^{-\nu n},
\end{equation}
for every $k$ such that $\vep_k \leq r_n^2$. Let us take $\alpha \in (0,1)$, $n_0 \in \NN$ and $C > 0$ satisfying
\[
\alpha < \nu, \qquad 1 - \frac{\theta_0}{2} \leq 4^{-\alpha}, \qquad 4^{\nu n_0} \geq \tfrac{1}{\theta_0\dd_0}, \qquad C = 2\cdot 4^{\alpha n_0}.
\]
Then, if $n = n_0$, \eqref{eq:OscDecayFinal} holds true by definition of $C$, recalling that $\|V_k\|_{L^\infty(\QQ_4)} \leq 1$, for every $k$. Now, assume that \eqref{eq:OscDecayFinal} holds true for some $n\geq n_0$. Then, by definition of $n_0$, $\alpha$, $C$ and \eqref{eq:OscDecayCalpha} and the inductive assumption, we obtain
\[
\begin{aligned}
\osc_{\QQ_{4^{-n}}} V_k &\leq (1 - \theta_0) \osc_{\QQ_{4^{-n+1}}} V_k + \theta_0 4^{\nu(n_0 - n)} \leq (1 - \theta_0) C 4^{-\alpha n} + \theta_0 4^{\nu(n_0 - n)} \\
&\leq (1 - \theta_0) C 4^{-\alpha n} + \theta_0 4^{\alpha(n_0 - n)} = \big(1 - \tfrac{\theta_0}{2} \big) C 4^{-\alpha n} \leq C4^{-\alpha(n+1)},
\end{aligned}
\]
for every $k$ such that $\vep_k \leq r_{n+1}^2$, and so \eqref{eq:OscDecayFinal} follows.

Notice that, since $\vep_k \to 0$, we may extract a decreasing subsequence $\vep_n := \vep_{k_n}$ such that $\vep_n \leq r_n^2$, for every $n \in \NN$ and thus, by \eqref{eq:OscDecayFinal}, we have
\begin{equation}\label{eq:OscDecayFinalBis}
\osc_{\QQ_{4^{-n+1}}(X_0,t_0)} V_n \leq C 4^{-\alpha n}, 
\end{equation}
for every $(X_0,t_0) \in \QQ_1$ and every $n \geq n_0$, where $V_n := V_{k_n}$.

To complete the proof, we check that
\begin{equation}\label{eq:HolderBoundVk}
|V_n(X,t) - V_n(Z,\tau)| \leq C \|(X-Z,t-\tau)\|^\alpha,
\end{equation}
for every $(X,t),(Z,\tau) \in \QQ_1$, $n \in \NN$, taking eventually new constants $C > 0$ and $\alpha \in (0,1)$ depending only on $N$, $a$, $p$ and $q$. Since the constants $\alpha$, $C$ and $n_0$ in \eqref{eq:OscDecayFinalBis} are independent of $(X_0,t_0) \in \QQ_1$, it is enough to consider the case $(Z,\tau) = 0$.

Given $(X,t) \in \QQ_1$, there is $n$ (depending on $(X,t)$) such that $(X,t) \in \QQ_{4^{-n}}\setminus \QQ_{4^{-n-1}}$, that is $4^{-n-1} \leq \|(X,t)\| \leq 4^{-n}$. If $n \leq n_0$, then
\[
|V_n(X,t) - V_n(0)| \leq \osc_{\QQ_{4^{-n}}} V_n \leq 2\|V_n\|_{L^\infty(\QQ_{4^{-n}})} \leq 2 \leq 2 \cdot 4^{\alpha(n_0+1)} 4^{-\alpha(n+1)} \leq 4^\alpha C \|(X,t)\|^\alpha,
\]
where we have used that $\|V_n\|_{L^\infty(\QQ_1)} \leq 1$. When $n \geq n_0$ we directly apply \eqref{eq:OscDecayFinalBis} to deduce
\[
|V_n(X,t) - V_n(0)| \leq \osc_{\QQ_{4^{-n}}} V_n \leq C4^{-\alpha(n+1)}  \leq C \|(X,t)\|^\alpha,
\]
and \eqref{eq:HolderBoundVk} follows. Now, re-writing \eqref{eq:HolderBoundVk} in terms of $U_n := U_{k_n}$, we obtain
\[
\| U_n(X,t) - U_n(Z,\tau) \| \leq C \big( \delta + \| U_n \|_{L^\infty(\QQ_4)} + \|F_n\|_{L^{p,a}(\QQ_4)} + \|f_n\|_{L_\infty^q(Q_4)} \big) \|(X-Z,t-\tau)\|^\alpha,
\]
for every $n \in \NN$ and every $(X,t),(Z,\tau) \in \QQ_1$, where $F_n := F_{j_{k_n}}$ and $f_n := f_{j_{k_n}}$. The bound \eqref{eq:HolderBoundsEstFf} follows by letting $\delta \to 0$.
\end{proof}
\begin{rem}\label{rem:Translation2} As pointed out in Remark \ref{rem:Translation1}, the proofs of this section work for weak solutions in $\QQ_r(X_0,t_0)$ too, with minor modifications and constants independent of $(X_0,t_0) \in\RR^{N+2}$.
\end{rem}
%
%
%
%
%

%
%
%%%%%%%%%%%%%%%%%%%%%%%%%%%%%%%%%%%%%%%%%%%%%%%%%%%%%%%%%%%%%%%%%%%%%%%%%%%%%%%%%%%%%%%%%%%%%%%%%%%%%%%%%%%%%%
%
%%%%%%%%%%%%%%%%%%%%%%%%%%%%%%%%%%%%%%%%%%%%%%%%%%%%%%%%%%%%%%%%%%%%%%%%%%%%%%%%%%%%%%%%%%%%%%%%%%%%%%%%%%%%%%
%
%

%
\section{Proof of Theorem \ref{thm:MainIntro} and Corollary \ref{cor:MainIntro}}\label{Sec:FinalProofs}
\begin{proof}[Proof of Theorem \ref{thm:MainIntro}] Let $N \geq 1$, $a \in (-1,1)$, $\beta$ as in \eqref{eq:Reaction}, $U_0$ as in \eqref{eq:Ass2Data} and let  $\{U_\vep\}_{\vep \in (0,1)}$ be a family of minimizers of $\mathcal{F}_\vep$ defined in \eqref{eq:Functional}. 

By Lemma \ref{Lem:EulerEqMinFF} and Proposition \ref{prop:weaklimit}, there exist $U \in \UU_0$ and a sequence $\vep_j \to 0$ such that the pair $(U_j,u_j) := (U_{\vep_j},U_{\vep_j}|_{y=0})$ satisfies
\[
\varepsilon_j \int_{\mathbb{Q}_\infty} |y|^a \partial_t U_j \partial_t \eta \dX\rd t + \int_{\mathbb{Q}_\infty} |y|^a (\partial_t U_j \eta + \nabla U_j \cdot\nabla\eta) \dX\rd t + \int_{Q_\infty} \beta(u_j) \eta|_{y=0} \dx \rd t = 0,
\]
for every $\eta \in C_0^\infty(\mathbb{Q}_\infty)$ and every $j \in \NN$ and, further,
\begin{equation}\label{eq:StrongConvFinalThm}
U_j \to U \quad \text{ in } C_{loc}([0,\infty):L^{2,a}(\RR^{N+1}))
\end{equation}
as $j \to +\infty$. Now, let us fix $R > 0$, $t_0 := (8R)^2$ and consider the sequence $V_j(X,t) := U_j(RX,R^2t + t_0)$. Setting $f_j := \beta(u_j)$ (with $\|f_j\|_{L^\infty(Q_\infty)} \leq \|\beta\|_{L^\infty(\RR)}$ for every $j \in \NN$) and $f_{j,R}(x,t) := R^{1-a}f_j(Rx,R^2t)$, we have by Remark \ref{rem:ScalingProp}
\[
\int_{\mathbb{Q}_8} |y|^a ( \tfrac{\varepsilon_j}{R^2} \partial_t V_j \partial_t \eta + \partial_t V_j \eta + \nabla V_j \cdot\nabla\eta ) \dX\rd t - \int_{Q_8} f_{j,R} \eta|_{y=0} \dx \rd t = 0,
\]
for every $\eta \in C_0^\infty(\QQ_8)$ and every $j \in \NN$. Consequently, by \eqref{eq:StrongConvFinalThm}, we may apply Proposition \ref{thm:CalphaBound} and Proposition \ref{prop:L2LinfEstimate} to the sequence $\{V_j\}_{j\in\NN}$ to deduce 
\[
\| V_{j_k} \|_{C^{\alpha,\alpha/2}(\QQ_1)} \leq C \big( \| V_{j_k} \|_{L^2(\QQ_8)} + \|f_{j_k,R}\|_{L^\infty(Q_8)} \big),
\]
for every $k \in \NN$, some $C > 0$ and $\alpha \in (0,1)$ depending only on $N$ and $a$, and some sequence $j_k \to +\infty$ (depending also on $R$). Re-writing this uniform bound in terms of $U_{j_k}$ and recalling the definition of $f_{j_k}$, we find
\begin{equation}\label{eq:HolderBoundFinal}
\| U_{j_k} \|_{C^{\alpha,\alpha/2}(\tilde{\QQ}_R)} \leq C R^{-\alpha} \Big[ R^{-\frac{N+3+a}{2}} \| U_{j_k} \|_{L^2(\QQ_{8R}^+)} + R^{1-a}\|\beta\|_{L^\infty(\RR)}  \Big],
\end{equation}
for every $k \in \NN$, where $\tilde{\QQ}_R := \BB_R \times (63 R^2,65 R^2)$. Finally, since
\[
\| U_{j_k} \|_{L^2(\QQ_{8R}^+)} \leq CR^2,
\]
uniformly in $k$ by Proposition \ref{thm:UnifEnergyEst} and Poincar\'e inequality (up to taking $C$ larger depending on $U_0$), we can combine Remark \ref{rem:Translation1}, Remark \ref{rem:Translation2} and a standard covering argument to complete the proof of  \eqref{eq:UnifHoldParIntro}.
\end{proof}
\begin{proof}[Proof of Corollary \ref{cor:MainIntro}] The thesis follows by combining Proposition \ref{prop:weaklimit}, Theorem \ref{thm:MainIntro} and the Arzel\`a-Ascoli theorem.
\end{proof}

\newpage

%%%%%%%%%%%%%%%%%%%%%%%%%%%%%%%%%%%%%%%%%%%%%%%%%%%%%%%%%%%%%%%%%%%%%%%%%%%%%%%%%%%%%%%%%%%%%%%%%%%%%%%%%%%%%

% APPENDIX

%%%%%%%%%%%%%%%%%%%%%%%%%%%%%%%%%%%%%%%%%%%%%%%%%%%%%%%%%%%%%%%%%%%%%%%%%%%%%%%%%%%%%%%%%%%%%%%%%%%%%%%%%%%%%

\normalcolor
\appendix

\section{}\label{App:WeightedSpaces}
In this section we review the definitions and some well-known properties of a class of energy spaces we use in our functional setting. The references for this part are \cite{ChiarenzaSerapioni1985:art,FabesKS1982:art,Hajlasz1996:art,JinLiXiong:art}. The symbols $I$, $\BB$ and $\QQ$ denote a generic interval in $\RR$, a generic ball in $\RR^{N+1}$, and a generic parabolic cylinder in $\RR^{N+2}$, respectively.

For $p \in (1,\infty)$ and $a \in (-1,1)$, we define
\[
\|U\|_{L^{p,a}(\BB)} := \Big( \int_\BB |y|^a|U|^p \dX \Big)^{\frac{1}{p}}, \qquad \|U\|_{L^{p,a}(\QQ)} := \Big( \int_\QQ |y|^a|U|^p \dX \rd t \Big)^{\frac{1}{p}}.
\]
We denote with $L^{p,a}(\BB)$ the closure of $C_0^\infty(\BB)$ w.r.t. the norm $\|\cdot\|_{L^{p,a}(\BB)}$, and $L^{p,a}(\QQ)$ the closure of $C_0^\infty(\QQ)$ w.r.t. the norm $\|\cdot\|_{L^{p,a}(\QQ)}$. Since the weight $y \to |y|^a \in L_{loc}^1(\RR)$, it is not difficult to show that
\[
\|U\|_{L^\infty(\BB)} := \esssup_{X \in \BB_1} |U(X)| = \lim_{p\to+\infty} \|U\|_{L^{p,a}(\BB)}.
\]
We set
\[
\begin{aligned}
\|U\|_{H^{1,a}(\BB)} &:= \Big( \int_\BB |y|^a U^2 \dX+ \int_\BB |y|^a |\nabla U|^2 \rd X \Big)^{\frac{1}{2}} \\
\|U\|_{H^{1,a}(\QQ)} &:= \Big( \int_\QQ |y|^a U^2 \dX\rd t + \int_\QQ |y|^a \left(|\partial_tU|^2 + |\nabla U|^2 \right)\rd X\rd t \Big)^{\frac{1}{2}}.
\end{aligned}
\]
The space $H^{1,a}(\BB)$ is the closure of $C^\infty(\BB)$ w.r.t. the norm $\|\cdot\|_{H^{1,a}(\BB)}$, while $H_0^{1,a}(\BB)$ denotes the closure of $C_0^\infty(\BB)$ w.r.t. the norm $\|\cdot\|_{H^{1,a}(\BB)}$. The spaces $H^{1,a}(\QQ)$ and $H_0^{1,a}(\QQ)$ are defined analogously, while $H^{-1,a}(\BB)$ is the dual space of $H_0^{1,a}(\BB)$.

Finally, we set
\[
\begin{aligned}
\|U\|_{L^2(I:L^{2,a}(\BB))} &:= \Big( \int_I \|U(t)\|_{L^{2,a}(\BB)}^2 \rd t \Big)^{\frac{1}{2}} \\
\|U\|_{L^2(I:H^{1,a}(\BB))} &:= \Big( \int_I \|U(t)\|_{H^{1,a}(\BB)}^2 \rd t \Big)^{\frac{1}{2}},
\end{aligned}
\]
and we define the spaces $L^2(I:L^{2,a}(\BB))$ and $L^2(I:H^{1,a}(\BB))$ as the closure of the space $C^\infty(\BB\times I)$ w.r.t. the norms $\|\cdot\|_{L^2(I:L^{2,a}(\BB))}$ and $\|\cdot\|_{L^2(I:H^{1,a}(\BB))}$, respectively.

As it is well-known, these spaces enjoy some notable properties that we resume below.
\begin{thm}\label{thm:TraceThm} (cf. \cite[Proposition 2.1 and Proposition 2.6]{JinLiXiong:art})

\noindent Let $N \geq 1$, $a \in (-1,1)$. There exist a unique bounded linear operator $|_{y=0} : H^{1,a}(\BB_1) \to H^{\frac{1-a}{2}}(B_1)$ and a constant $C > 0$ depending only on $N$ and $a$, such that if $u = U|_{y=0}$ then
\[
\int_{B_1} u^2 \rd x + \iint_{B_1\times B_1} \frac{(u(x) - u(z))^2}{|x-z|^{N+1-a}} \rd x\rd z \leq C \Big( \int_{\mathbb{B}_1} |y|^a U^2 \rd X +  \int_{\mathbb{B}_1} |y|^a |\nabla U|^2 \rd X  \Big)
\]
for every $U \in H^{1,a}(\BB_1)$. Furthermore, we have
\begin{equation}\label{eq:TraceIneqWithEps}
\int_{B_1} u^2 \rd x \leq C \Big( A^{\frac{1+a}{2}} \int_{\mathbb{B}_1} |y|^a U^2 \rd X +  A^{-\frac{1-a}{2}} \int_{\mathbb{B}_1} |y|^a |\nabla U|^2 \rd X  \Big),
\end{equation}
for every $A > 1$. Finally, we also have
\begin{equation}\label{eq:EmbeddingTrace}
\Big(\int_{B_1} |u|^{2q} \rd x \Big)^{\frac{1}{q}} \leq C \int_{\mathbb{B}_1} |y|^a |\nabla U|^2 \rd X,
\end{equation}
for every $U \in H^{1,a}_0(\BB_1)$ and every $q \in [1,\tilde{\sigma}]$, where
\[
\tilde{\sigma} := \frac{N}{N-1+a} = 1 + \frac{1-a}{N-1+a}.
\]
\end{thm}
\begin{thm}\label{thm:SobEmbedLocalEll} (\cite[Theorem 1.2]{FabesKS1982:art} and \cite[Theorem 6, $p=2$, $s=N+1+a$]{Hajlasz1996:art})

\noindent Assume either $N \geq 2$ and $a \in (-1,1)$, or $N=1$ and $a \in [0,1)$. Then there exists a constant $\mathcal{S}_0 > 0$ depending only on $N$ and $a$ such that
\[
\|U\|_{L^{2\sigma,a}(\mathbb{B}_r)}^2  \leq \mathcal{S}_0 \, \big[ \, \tfrac{1}{r^2} \|U\|_{L^{2,a}(\mathbb{B}_r)}^2 +  \|\nabla U\|_{L^{2,a}(\mathbb{B}_r)}^2  \big],
\]
for all $r > 0$ and all $U \in H^{1,a}(\mathbb{B}_r)$, where
\[
\sigma := 
\begin{cases}
+\infty \quad &\text{ if } N=1, \; a = 0 \\
1 + \frac{2}{N-1+a} \quad &\text{ otherwise}.
\end{cases}
\]
Further, if $N=1$ and $a \in (-1,0)$, then\footnote{This second inequality follows by \cite[Theorem 6, part 3]{Hajlasz1996:art}, proceeding as in \cite[Theorem 9.12]{Brezi2019:art}.}
\[
\|U\|_{L^\infty(\mathbb{B}_r)}^2  \leq \mathcal{S}_0 \, r^{|a|} \big[ \, \tfrac{1}{r^2} \|U\|_{L^{2,a}(\mathbb{B}_r)}^2 +  \|\nabla U\|_{L^{2,a}(\mathbb{B}_r)}^2  \big] 
\]
for all $r > 0$ and all $U \in H^{1,a}(\mathbb{B}_r)$.
\end{thm}
\begin{thm}\label{thm:SobEmbedLocalPar} (\cite[Lemma 2.1]{ChiarenzaSerapioni1985:art})

\noindent Assume either $N \geq 2$ and $a \in (-1,1)$, or $N=1$ and $a \in [0,1)$. Then there exists a constant $\mathcal{S} > 0$ depending only on $N$ and $a$ such that
\begin{equation}\label{eq:ParSobIneq1}
\int_{\mathbb{Q}_r} |y|^a |U|^{2\gamma} \rd X\rd t
\leq \mathcal{S} \Big( \tfrac{1}{r^2}\int_{\mathbb{Q}_r} \!|y|^a U^2 \rd X\rd t + \int_{\mathbb{Q}_r} \!|y|^a |\nabla U|^2 \rd X\rd t \Big) \,\esssup_{t \in (-r^2,r^2)} \Big( \int_{\mathbb{B}_r} \!|y|^a U^2 \rd X \Big)^{\gamma -1}
\end{equation}
for all $r > 0$ and all $U \in L^2(-r^2,r^2 : H^{1,a}(\mathbb{B}_r))$, where
\[
\gamma := \frac{2\sigma - 1}{\sigma} = 1 + \frac{2}{N + 1 + a}.
\]
Further, if $N=1$ and $a \in (-1,0)$, then
\begin{equation}\label{eq:ParSobIneq2}
\int_{\mathbb{Q}_r} |y|^a U^4 \rd X\rd t
\leq \mathcal{S} r^{|a|}\Big( \tfrac{1}{r^2}\int_{\mathbb{Q}_r} \!|y|^a U^2 \rd X\rd t + \int_{\mathbb{Q}_r} \!|y|^a |\nabla U|^2 \rd X\rd t \Big) \,\esssup_{t \in (-r^2,r^2)} \Big( \int_{\mathbb{B}_r} \!|y|^a U^2 \rd X \Big)
\end{equation}
for all $r > 0$ and all $U \in L^2(-r^2,r^2 : H^{1,a}(\mathbb{B}_r))$.

\end{thm}
\begin{proof} It is enough to prove the statement in the case $r=1$, $N \geq 2$ and $a \in (-1,1)$ (the other cases are similar). Let $U \in L^2(-1,1 : H^{1,a}(\mathbb{B}_1))$, $\sigma$ as in Theorem \ref{thm:SobEmbedLocalEll}, $\gamma = (2\sigma-1)/\sigma$ and $\sigma' = \sigma/(\sigma - 1)$. Assume first $\sigma < +\infty$. For a.e. $t \in (0,1)$, we have
\[
\begin{aligned}
\int_{\mathbb{B}_1} |y|^a |U|^{2\gamma} \dX &= \int_{\mathbb{B}_1} |y|^a U^2 |U|^{2/\sigma'} \dX \leq \Big(\int_{\mathbb{B}_1} |y|^a |U|^{2\sigma} \dX \Big)^{1/\sigma} \Big(\int_{\mathbb{B}_1} |y|^a U^2 \dX \Big)^{1/\sigma'} \\
&
\leq C \Big( \int_{\mathbb{B}_1} |y|^a U^2 \dX + \int_{\mathbb{B}_1} |y|^a |\nabla U|^2 \dX \Big) \esssup_{t \in (-1,1)} \Big( \int_{\mathbb{B}_1} |y|^a U^2 \dX \Big)^{\gamma -1},
\end{aligned}
\]
by H\"older inequality and Theorem \ref{thm:SobEmbedLocalEll}. The thesis follows integrating in time.
\end{proof}
%

%
%
%%%%%%%%%%%%%%%%%%%%%%%%%%%%%%%%%%%%%%%%%%%%%%%%%%%%%%%%%%%%%%%%%%%%%%%%%%%%%%%%%%%%%%%%%%%%%%%%%%%%%%%%%%%%%%%%%%%%%%%%%%%%%%%%%%%%%%%%%%%%%%%%%%%%%
%
%
%\subsection{Notations}\label{Sec:Notations}
%
%
\section{}\label{App:TechnicalLemma}
In this second appendix we show a technical lemma. The proof is quite standard, but we include it for completeness.
\begin{lem}\label{lem:IneqTruncations} Let $N \geq 1$, $a \in (-1,1)$, $R > 0$, $\vep \in (0,1)$, $(p,q)$ satisfying \eqref{eq:AssumptionsExponents}, $F_\vep \in L^{p,a}(\QQ_R)$ and $f_\vep \in L_\infty^q(Q_R)$. Then, the following statements hold true:

(i) If $U_\vep$ is a weak solution in $\QQ_R$, then for every $l \in \RR$, the functions
\[
V_+ := (U_\vep - l)_+ \quad \text{ and } \quad V_- := (U_\vep - l)_-
\]
are weak subsolutions in $\QQ_R$, with $F_\vep$ and $f_\vep$ replaced by $(F_\vep)_+$ and $(f_\vep)_+$, and $(F_\vep)_-$ and $(f_\vep)_-$, respectively.

(ii) If $U_\vep$ is a weak subsolution in $\QQ_R$ and $F_\vep,f_\vep \geq 0$, then for every $l \geq 0$, the function
\[
V = \max\{U,l\}
\]
is a weak subsolution in $\QQ_R$.
\end{lem}
\begin{proof} We give a sketch of the proof of (i) (part (ii) follows analogously). Let $U = U_\vep$, $F = F_\vep$, $f= f_\vep$ and $l \in \RR$. Since $U-l$ is still a weak solution in $\QQ_R$, it is enough to consider the case $l = 0$.

Let $p(U) = U_+$ and consider a sequence of smooth functions $p_j:\RR \to \RR$ such that
\[
p_j,p_j',p_j'' \geq 0, \quad p_j = p \text{ in } \RR\setminus (-\tfrac{1}{j},\tfrac{1}{j}), \quad \|p_j-p\|_{H^{1,a}(\RR)} \leq \tfrac{1}{j},
\]
and let $V_j := p_j(U)$ for every integer $j \geq 1$. We fix a nonnegative $\eta \in C_0^\infty(\QQ_R)$ and consider
\[
\begin{aligned}
\int_{\mathbb{Q}_R} &|y|^a ( \varepsilon \partial_t V_j \partial_t \eta + \partial_t V_j \eta + \nabla V_j \cdot\nabla\eta - F_+ \eta) \dX\rd t - \int_{Q_R} f_+ \eta|_{y=0} \dx \rd t  \\
&= \int_{\mathbb{Q}_R} |y|^a ( \varepsilon p_j'(U) \partial_t U \partial_t \eta + p_j'(U) \partial_t U \eta + p_j'(U) \nabla U \cdot\nabla\eta - F_+ \eta) \dX\rd t - \int_{Q_R} f_+ \eta|_{y=0} \dx \rd t \\
&= \int_{\mathbb{Q}_R} |y|^a \big[ \varepsilon \partial_t U \partial_t(p_j'(U) \eta) + \partial_t U (p_j'(U)\eta) + p_j'(U) \nabla U \cdot\nabla(p_j'(U) \eta) \big] \dX\rd t \\
&\quad - \vep \int_{\QQ_R} (\partial_t U)^2 p_j''(U)\eta \dX\rd t - \int_{\QQ_R} |\nabla U|^2 p_j''(U)\eta \dX\rd t  - \int_{\mathbb{Q}_R} |y|^a F_+ \eta \dX\rd t - \int_{Q_R} f_+ \eta|_{y=0} \dx \rd t \\
&\leq \int_{\mathbb{Q}_R} |y|^a F p_j'(U)\eta \dX\rd t + \int_{Q_R} f p_j'(u)\eta|_{y=0} \dx \rd t - \int_{\mathbb{Q}_R} |y|^a F_+ \eta \dX\rd t - \int_{Q_R} f_+ \eta|_{y=0} \dx \rd t
\end{aligned}
\]
where we have used \eqref{eq:WeakEquation} with test $p_j'(U) \eta$ and that $\eta,p_j',p_j'' \geq 0$. Since
\[
\int_{\mathbb{Q}_R} |y|^a F p_j'(U)\eta \dX\rd t + \int_{Q_R} f p_j'(u)\eta|_{y=0} \dx \rd t \to \int_{\mathbb{Q}_R\cap\{U > 0\}} |y|^a F \eta \dX\rd t + \int_{Q_R\cap\{u > 0\}} f \eta|_{y=0} \dx \rd t,
\]
as $j \to +\infty$, thanks to the Lebesgue dominated convergence theorem (and trace theorem), we deduce that $U_+$ is a weak subsolution in $\QQ_R$ with $F_+$ and $f_+$ by passing to the limit as $j \to + \infty$.

To complete the proof, it is enough to notice that $-U$ is a weak solution in $\QQ_R$ with $-F$ and $-f$. Then $U_- = (-U)_+$ is a weak subsolution in $\QQ_R$ with $F_- = (-F)_+$ and $f_- = (-f)_+$.
\end{proof}
\begin{rem}
The same proof shows that if $U_\vep$ is a weak subsolution in $\QQ_R$, then for every $l \in \RR$, the function $(U_\vep - l)_+$ is a weak subsolution in $\QQ_R$, with $F_\vep$ and $f_\vep$ replaced by $(F_\vep)_+$ and $(f_\vep)_+$, respectively.
\end{rem}

%
%
%%%%%%%%%%%%%%%%%%%%%%%%%%%%%%%%%%%%%%%%%%%%%%%%%%%%%%%%%%%%%%%%%%%%%%%%%%%%%%%%%%%%%%%%%%%%%%%%%%%%%%%%%%%%%%%%%%%%%%%%%%%%%%%%%%%%%%%%%%%%%%%%%%%%%
%
%
%\subsection{Notations}\label{Sec:Notations}
%
%
\section{}\label{App:Notations}
We report below the list of notations we use in the paper.
\[
\begin{aligned}
&I \; \text{ denotes a generic interval in } \mathbb{R} \\
& X = (x,y) \\
& B_r(x_0) = \{x \in \mathbb{R}^N: |x-x_0|^2 < r^2 \} \\
& B \; \text{ denotes a generic ball in } \mathbb{R}^N \\
& |\BB|_a := \int_\BB|y|^a\dX \\
& \mathbb{B}_r(X_0) = \{(x,y) \in \mathbb{R}^{N+1}: |x-x_0|^2 + |y-y_0|^2 < r^2 \} \\
& \mathbb{B} \; \text{ denotes a generic ball in } \mathbb{R}^{N+1} \\
& Q_r(x_0,t_0) = B_r(x_0) \times \{0\} \times (t_0-r^2,t_0+r^2) \quad (\text{parabolic cylinder in } \mathbb{R}^N\times\{0\}\times \RR )\\
& Q \; \text{ denotes a generic parabolic cylinder in } \mathbb{R}^N\times\{0\}\times(0,\infty) \\
&Q_r^+(x_0,t_0) := B_r(x_0) \times \{0\} \times (t_0,t_0+r^2) \\
&Q_\infty = \mathbb{R}^N  \times \{0\}\times (0,\infty) \\
& \mathbb{Q}_r(X_0,t_0) = \mathbb{B}_r(X_0) \times (t_0-r^2,t_0+r^2) \quad (\text{parabolic cylinder in } \mathbb{R}^{N+1}\times(0,\infty) ) \\
& \mathbb{Q}_r^+(X_0,t_0) = \mathbb{B}_r(X_0) \times (t_0,t_0+r^2) \\
& \mathbb{Q}_\infty = \mathbb{R}^{N+1}\times (0,\infty) \\
& \mathbb{Q} \; \text{ denotes a generic parabolic cylinder in } \mathbb{R}^{N+2} \\
& |\QQ|_a := \int_\QQ|y|^a\dX\rd t \\
& \mathbb{R}^{N+1}_+ := \mathbb{R}^N \times \{y > 0\} \\
& \osc_{\QQ} U := \esssup_\QQ U - \essinf_\QQ U \\
& [U]_{C^{\alpha,\alpha/2}(\QQ)} := \sup_{\substack{(X,t),(Y,\tau) \in \QQ \\ (X,t)\not=(Y,\tau)}} \frac{|U(X,t) - U(Y,\tau)|}{\|(X-Y,t-\tau)\|^\alpha}, \qquad\quad \|(Z,s)\| := \max\{|Z|,\sqrt{|s|}\}, \quad \alpha \in (0,1) \\
& \|U\|_{C^{\alpha,\alpha/2}(\QQ)} := \|U\|_{L^\infty(\QQ)} + [U]_{C^{\alpha,\alpha/2}(\QQ)}, \qquad \alpha \in (0,1).
\end{aligned}
\]
%

%%%%%%%%%%%%%%%%%%%%%%%%%%%%%%%%%%%%%%%%%%%%%%%%%%%%%%%%%%%%%%%%%%%%%%%%%%%%%%%%%%%%%%%%%%%%%%%%%%%%%%%%%%%%%

% REFERENCES

%%%%%%%%%%%%%%%%%%%%%%%%%%%%%%%%%%%%%%%%%%%%%%%%%%%%%%%%%%%%%%%%%%%%%%%%%%%%%%%%%%%%%%%%%%%%%%%%%%%%%%%%%%%%%

\end{document}